\documentclass[letterpaper,12pt]{extarticle}
\usepackage[margin=1.4in,letterpaper]{geometry} 
\usepackage{eucal}
\usepackage{mathrsfs}
\usepackage{subfigure}
\usepackage{listings}
\usepackage[dvips]{graphicx}
\usepackage{fancyhdr}
\usepackage{color}
\usepackage{theoremref, hyperref}
\usepackage{amsmath, amsthm}
\usepackage{amssymb}
\usepackage{tikz}
\usepackage{bm}
\usepackage{url}
\usepackage{latexsym}
\usepackage{arydshln}
\usepackage{amsfonts}
\usepackage{comment}
\usepackage{todonotes}[disable]

\numberwithin{equation}{section}

\newtheorem{remark}{Remark}[section]
\newtheorem{theorem}[remark]{Theorem}

\newtheorem{corollary}[remark]{Corollary}
\newtheorem{lemma}[remark]{Lemma}
\newtheorem{proposition}[remark]{Proposition}
\newtheorem{definition}[remark]{Definition}
\newtheorem{example}[remark]{Example}

\newcommand{\Addresses}{{
    \bigskip
    \footnotesize
    
    Monika D\"orfler, \textsc{Department of Mathematics, University of Vienna, 1090 Vienna, Austria}\par\nopagebreak
    \textit{E-mail address}: \texttt{monika.doerfler@univie.ac.at}
    
 Lukas K\"ohldorfer, \textsc{Acoustics Research Institute, Austrian Academy of Sciences, 1010 Vienna, Austria}\par\nopagebreak
    \textit{E-mail address}: \texttt{lukas.koehldorfer@oeaw.ac.at}
    
    Franz Luef, \textsc{Department of Mathematical Sciences, Norwegian University of Science and Technology, 7034
    Trondheim, Norway}\par\nopagebreak
    \textit{E-mail address}: \texttt{franz.luef@ntnu.no}
    
    \medskip
    
    Henry McNulty, \textsc{Cognite AS, 1366 Lysaker, Norway}\par\nopagebreak
    \textsc{Department of Mathematical Sciences, Norwegian University of Science and Technology, 7491
    Trondheim, Norway}\par\nopagebreak
    \textit{E-mail address}: \texttt{henry.mcnulty@cognite.com}
}}

\newcommand{\Hil}[0]{{\mathcal{H}}}
\newcommand{\B}[0]{{\mathcal{B}}}

\newcommand{\A}[0]{{\mathcal{A}}}
\newcommand{\J}[0]{{\mathcal{J}}}

\newcommand{\HS}[0]{{\mathcal{HS}}}

\newcommand{\Co}[0]{\mathcal{C}o}

\title{Approximation properties of operator coorbit spaces and sparsity classes}
\author{Monika Dörfler
\and Lukas Köhldorfer
\and Franz Luef
\and Henry McNulty 
  }
\begin{document}

\maketitle

\begin{abstract}
Extensions of coorbit spaces for functions to operators have been introduced by two different groups in \cite{doelumcskr24} and \cite{köbaLOC25}, where one is based on the coorbit theory of Feichtinger-Gr\"ochening while the other is based on the theory of localized frames. We show that for certain Gabor g-frames the co-orbit spaces in \cite{köbaLOC25} conincide with the ones in \cite{doelumcskr24} and we refer to this class of operators as operator coorbit spaces. Based on the description of operator coorbit spaces in terms of Gabor g-frames we provide operator dictionaries for these spaces that allow us to define sparsity classes in this setting. We establish that these sparsity classes also coincide with the operator coorbit spaces, which holds, in particular for all Feichtinger operators, a nice class of mixed states. Numerical examples confirm the expected approximation quality by few terms for appropriately chosen operators. 
\end{abstract}

\section{Introduction}


In applied harmonic analysis compressed sensing provides us with fast and reliable algorithms based on sparse representations of signals or images. As an example, wavelet decompositions provide sparse representations of images. In the search for dictionaries which enable sparse representations for a particular class of objects, it has been noticed that \emph{localized} frames can lead to sparsity in representations of functions in certain Banach spaces such as modulation spaces or Besov spaces, or more general coorbit spaces of functions associated to projective square-integrable, integrable unitary representations of a locally compact group $G$ on a Hilbert space $\Hil$, cf.~\cite{gr04-1,forngroech05,raul11}.

Many problems in quantum technology deal with quantum systems that have a huge number of quantum states, e.g. ground state problems, see,~e.g.,~ \cite{KuengPreskil2022}, where the authors propose the use of machine learning techniques to address this crucial issue. One of our main contributions to the theory of quantum states is to define sparsity classes and to construct {\it sparse} decompositions of mixed-state operators, in a sense made more precise in Section \ref{Approximation theoretic results}. These results are operator analogs of the ones on sparse decomposition for functions. Like in the case of functions, we introduce sparsity classes for co-orbit spaces of operators, defined and studied in \cite{doelumcskr24}. 

From a mathematical point of view, quantum states are positive trace class operators with unit trace, aka density operators, and in many situations we are dealing with a finite-rank operator, which is the sum of a large number of rank-one operators. Another example of an operator of this form is the data operator introduced in~\cite{Dorf21}, which may be interpreted as the sample covariance operator of a random variable taking values in a Hilbert space $\mathcal{H}$,~\cite{taovu12}.

Given the importance of density operators in different contexts as described above, it is one of the goals of this work to develop sparse decompositions for density operators using operator-valued frames. In analogy to the function setting, these frames are generated by {\it operator translations} of a ``nice" trace class operator, the operator analysis window $\Phi$,  and where the analysis mapping is sampling the operator short-time Fourier transform,~\cite{doelumcskr24}, of the density operator with respect to the analysis trace class operator $\Phi$.

We show that results on sparsity classes associated with operator-valued frame yield as a consequence sparse representations of density operators in terms of the first few largest frame coefficients of the operator-valued frame decomposition of the density operator as in the case of best $m$-term approximation of signals. 
The results are based on the theory of localized operator-valued frames \cite{köbaLOC25}, which contain Gabor g-frames \cite{skrett21} as a special case, and a class of operators associated with the operator STFT in \cite{doelumcskr24}, which one might call operator coorbit spaces. Gabor g-frames have also been independently introduced by Han, Li and Wang in \cite{HanLiWang-OpValGaborFrames} under the name of operator-valued Gabor frames, motivated by the work of Kaftal, Larson and Zhang on operator-valued frames \cite{KaftalLarsonZhang-OpValuedFrames}. The sparse representation of density operators is in terms of tight operator Gabor g-frame decomposition of a suitable Gabor g-frame operator window. We conclude this introductory section by two important remarks. First, the existence of sparse operator representation presented in this article depends intrinsically on the localization of the representation system. This is in line with earlier work on sparse signal representation, cf.~\cite{grib-nielsen2}. Second, we  point out that  the proposed sparse representations are related to "pretty good" measurements in quantum information theory, cf.~\cite{Hausladen1996,HausladenWootters1994,Gilyen2022,KotowskiOszmaniec2025}.

\section{Preliminaries}


Throughout this manuscript, $\Hil$ will always be a separable Hilbert space. By $\mathbb{N}$ we denote the set of positive intergers and we abbreviate $\mathbb{N}_0 = \mathbb{N} \cup \lbrace 0 \rbrace$.

\subsection{Weight functions}\label{Weight functions}

A \emph{weight function} on $\mathbb{R}^{d}$, or simply \emph{weight}, is a continuous and positive function $\nu :\mathbb{R}^{d} \longrightarrow (0,\infty )$.
The standard weight considered in this work 
is the polynomial weight $\nu_s$ ($s\geq 0$) given by 
$$\nu_s (z) = (1+\vert z \vert)^s \qquad (z\in \mathbb{R}^d).$$
A weight $m$ is called \emph{$\nu$-moderate} if there exists some constant $C>0$ such that 
$$m(z+z') \leq Cm(z)\nu(z') \qquad (\forall z, z' \in \mathbb{R}^{d}).$$

For more background on weight functions and their relevance in time-frequency analysis, see \cite{gr06weightsinTFA}.



\subsection{Bochner spaces}

Let $B$ be a Banach space, $1\leq p < \infty$ and $\nu$ be a weight on $\mathbb{R}^d$. Then the (weighted) \emph{Lebesgue-Bochner space} \cite[Chapter 1]{HyNeVeWe16} $L^p_{\nu}(\mathbb{R}^d;B)$ denotes the space of all (equivalence classes of) strongly Lebesgue-measurable functions $f:\mathbb{R}^d \longrightarrow B$ for which 
$$\int_{\mathbb{R}^d} \Vert f(z)\Vert_B^p \nu(z)^p\, dz $$
is finite. Equipped with the norm 
$$\Vert f \Vert_{L^p_{\nu}(\mathbb{R}^d;B)}:= \left( \int_{\mathbb{R}^d} \Vert f(z)\Vert_B^p \nu(z)^p \, dz \right)^{\frac{1}{p}}$$
$L^p_{\nu}(\mathbb{R}^d;B)$ is a Banach space. The Banach space $L^{\infty}_{\nu}(\mathbb{R}^d;B)$ is defined via the usual modifications. In case $\nu \equiv 1$ is the trivial weight on $\mathbb{R}^d$ we abbreviate $L^p(\mathbb{R}^d;B) := L^p_1(\mathbb{R}^d;B)$ ($1\leq p \leq \infty$).

Next, we consider the discrete variants of the Lebesgue-Bochner spaces introduced above. Let again $B$ be a Banach space, $\nu$ be a weight on $\mathbb{R}^d$ and let $\Lambda$ be a discrete subset of $\mathbb{R}^d$. Then for $p\in (0,\infty]$ the (weighted) \emph{Bochner sequence space} $\ell^p_{\nu}(\Lambda; B)$ \cite[Chapter 1]{HyNeVeWe16} is defined by 
$$\ell_{\nu}^p(\Lambda; B) := \left\lbrace (h_k)_{k\in \Lambda}: h_k \in B \, (\forall k\in \Lambda), (\Vert h_k \Vert_B \cdot \nu(k))_{k\in \Lambda} \in \ell^p(\Lambda) \right\rbrace .$$
In case $\nu(k) = 1$ for all $k\in \Lambda$, we write $\ell_{\nu}^p(\Lambda; B) = \ell^p(\Lambda; B)$. By \cite[Chapter 1]{HyNeVeWe16}, $\ell_{\nu}^p(\Lambda; B)$ equipped with the (quasi-)norm 
$$\Vert (h_k)_{k\in \Lambda} \Vert_{\ell_{\nu}^p(\Lambda; B)} := \Vert (\Vert h_k \Vert_B \cdot \nu(k))_{k\in \Lambda}\Vert_{\ell^p(\Lambda)},$$
is a Banach space for every $1\leq p \leq \infty$ and a quasi-Banach space for $0<p<1$. By $c_{00}(\Lambda ;B)$ we denote the space of all finitely supported $B$-valued sequences indexed by $\Lambda$. It is straightforward to see that $(c_{00}(\Lambda;B), \Vert \, . \, \Vert_{\ell_{\nu}^p(\Lambda; B)})$ is a dense subspace of $(\ell_{\nu}^p(\Lambda;B), \Vert \, . \, \Vert_{\ell_{\nu}^p(\Lambda; B)})$ for each $p\in [1,\infty)$.

\subsection{Some background on g-frames}\label{Some background on g-frames and localization}
Frames, introduced by Duffin and Schaeffer in 1952 \cite{duffschaef1}, are countable families of vectors in a separable Hilbert space, which allow a redundant and stable reconstruction of any vector in that space from linear rank-one measurements. Here we consider a generalization of frames, originally called \emph{generalized frames} (\emph{g-frames}) \cite{sun06}, that is, countable families of operators which allow a linear and stable reconstruction of any vector from higher-rank measurements. 

A countable family $(U_{\lambda})_{\lambda \in \Lambda}$ of bounded operators $U_{\lambda} \in \B(\Hil)$ on some separable Hilbert space $\Hil$ is called an \emph{operator-valued frame} \cite{sun06} for $\Hil$, if there exist positive constants $0<A\leq B<\infty$, such that 
\begin{equation}\label{gframedef}
    A\Vert f \Vert^2 \leq \sum_{\lambda \in \Lambda} \Vert U_{\lambda} f \Vert^2 \leq B\Vert f \Vert^2 \qquad(\forall f\in \Hil).
\end{equation}
For brevity, we will stick to the name \emph{g-frame} from now on and call $A$ and $B$ the \emph{lower} and \emph{upper g-frame bound} respectively. We call $(U_{\lambda})_{\lambda \in \Lambda}$ a \emph{g-Bessel sequence} with \emph{Bessel bound} $B$, whenever the upper (but not necessarily the lower) inequality in (\ref{gframedef}) holds true for all $f\in \Hil$. For any g-Bessel sequence $\mathcal{U} =(U_{\lambda})_{\lambda \in \Lambda}$ the following operators are well-defined and bounded \cite{sun06}:
\begin{itemize}
    \item The \emph{synthesis operator} $D_{\mathcal{U}} : \ell^2(\Lambda;\Hil) \longrightarrow \mathcal{H}$, defined by $$D_{\mathcal{U}} (g_{\lambda})_{\lambda \in \Lambda} = \sum_{\lambda \in \Lambda} U^*_{\lambda} g_{\lambda} ,$$
    \item The \emph{analysis operator} 
    $C_{\mathcal{U}} : \mathcal{H} \longrightarrow \ell^2(\Lambda;\Hil)$, defined by $$C_{\mathcal{U}} f = ( U_{\lambda} f )_{\lambda \in \Lambda} ,$$
    \item The \emph{g-frame operator}
    $S_{\mathcal{U}} := D_{\mathcal{U}} C_{\mathcal{U}}: \mathcal{H} \longrightarrow \mathcal{H}$, given by $$S_{\mathcal{U}} f = \sum_{\lambda \in \Lambda} U^*_{\lambda} U_{\lambda} f .$$
    \item The \emph{g-Gram matrix} $G_{\mathcal{U}} := C_{\mathcal{U}} D_{\mathcal{U}} : \ell^2(\Lambda;\Hil) \longrightarrow \ell^2(\Lambda;\Hil)$.
\end{itemize}
In this case, $\Vert D_{\mathcal{U}} \Vert = \Vert C_{\mathcal{U}} \Vert \leq \sqrt{B}$ and $D_{\mathcal{U}}$ and $C_{\mathcal{U}}$ are adjoint to one another. Consequently, both $S_{\mathcal{U}}$ and $G_{\mathcal{U}}$ are self-adjoint and bounded by $B$. If $\mathcal{U}$ is a g-frame with frame bounds $A\leq B$, then $S_{\mathcal{U}}$ is additionally a positive operator, whose spectrum is contained in $[A,B]$, and thus invertible on $\Hil$. Hence $S_{\mathcal{U}}$ may be composed with its inverse (and vice versa), which yields the \emph{g-frame reconstruction} formulae   
\begin{equation}\label{gframerec}
f = \sum_{\lambda \in \Lambda} U^*_{\lambda} U_{\lambda} S_{\mathcal{U}}^{-1} f = \sum_{\lambda \in \Lambda} S_{\mathcal{U}}^{-1} U^*_{\lambda} U_{\lambda} f \qquad (\forall f\in \mathcal{H}) .
\end{equation}
The family $\widetilde{\mathcal{U}} = (\widetilde{U}_{\lambda})_{\lambda \in \Lambda}$, defined by $\widetilde{U}_{\lambda} = U_{\lambda} S_{\mathcal{U}}^{-1}$ ($\lambda \in \Lambda$), which appears in the reconstruction process (\ref{gframerec}), is again a g-frame and called the \emph{canonical dual g-frame} of $\mathcal{U}$. 
More generally \cite{kutpatphi17}, if $\mathcal{W} = (W_{\lambda})_{\lambda \in \Lambda}$ is a g-Bessel sequence so that 
\begin{equation}\label{dualgframe}
f = \sum_{\lambda \in \Lambda} U^*_{\lambda} W_{\lambda} f = \sum_{\lambda \in \Lambda} W^*_{\lambda} U_{\lambda} f \qquad (\forall f\in \mathcal{H}), 
\end{equation}
or, equivalently, 
\begin{equation}\label{dualgframeoperatornotation}
\mathcal{I}_{\B(\Hil)} = D_{\mathcal{U}} C_{\mathcal{W}} = D_\mathcal{W} C_{\mathcal{U}} , 
\end{equation}
then $(W_{\lambda})_{\lambda \in \Lambda}$ is already g-frame and called a \emph{dual g-frame} of $(U_{\lambda})_{\lambda \in \Lambda}$. 

\subsection{Polynomially localized g-frames}\label{Polynomially localized g-frames} 

Besides the above-mentioned possibility of perfect reconstruction (\ref{gframerec}), some g-frames admit other useful properties which are not captured by the defining g-frame inequalities (\ref{gframedef}) alone. 

A g-frame $\mathcal{U} = (U_{\lambda})_{\lambda \in \Lambda}$ for $\Hil$ is called \emph{polynomially $s$-localized} ($s>0$), or simply \emph{$s$-localized}, if there exists a constant $C>0$, such that 
\begin{equation}\label{slocalized}
\Vert U_{\lambda} U_{\mu}^* \Vert_{\B(\Hil)} \leq C (1+\vert \lambda - \mu \vert)^{-s} \qquad (\forall \lambda , \mu \in \Lambda).    
\end{equation}
Furthermore, if $\mathcal{W} = (W_{\lambda})_{\lambda \in \Lambda}$ is another g-frame for $\Hil$, then we say that $\mathcal{U}$ and $\mathcal{W}$ are \emph{mutually (polynomially) $s$-localized} ($s>0$), if there exists a constant $C>0$, such that 
\begin{equation}\label{mutually}
\Vert U_{\lambda} W_l^* \Vert_{\B(\Hil)} \leq C (1+\vert \lambda - \mu \vert)^{-s} \qquad (\forall \lambda , \mu \in \Lambda).    
\end{equation}

\begin{remark}
Relations (\ref{slocalized}) and (\ref{mutually}) say that both the g-Gram matrix $G_{\mathcal{U}} = [U_{\lambda} U_{\mu}^*]_{\lambda , \mu\in \Lambda}$ and the mixed g-Gram matrix $G_{\mathcal{U},\mathcal{W}} = [U_{\lambda} W_l^*]_{\lambda , \mu\in \Lambda}$ belong to the \emph{Jaffard class} $\J_s$ of $\B(\Hil)$-valued matrices with polynomial off-diagonal decay of order $s$. In particular, if the index set $\Lambda$ is assumed to be a \emph{relatively separated} set in $\mathbb{R}^d$, that is 
$$\sup_{x\in \mathbb{R}^d} \vert \Lambda \cap (x+[0,1]^d) \vert < \infty,$$
and if $s>d$, then $\J_s$ is a so-called \emph{spectral algebra} \cite{koeba25}, i.e. a Banach algebra of $\B(\Hil)$-valued matrices with manifold convenient properties, see also \cite{köbasampta25}.           
\end{remark}

Polynomially localized g-frames and their rich properties have been studied in \cite{köbaLOC25}. In the following, we provide the definitions and results relevant to this work. For the remainder of this section, we always assume the following. 

\

\noindent \textbf{Assumption:} $\Lambda \subset \mathbb{R}^d$ is a relatively separated (index) set. 

\


Later we will only work with a full-rank lattice $\Lambda$, which satisfies the latter (necessary) technical condition anyway.  

First, we note the remarkable fact, that $s$-localization is preserved under canonical duality.

\begin{theorem}\label{duallocalized}\cite{köbaLOC25}
Let $s>d$ and $\mathcal{U} =(U_{\lambda})_{\lambda \in \Lambda}$ be an $s$-localized g-frame. Then the following hold:
\begin{itemize}
    \item[(i)] The canonical dual g-frame $\widetilde{\mathcal{U}} =(\widetilde{U}_{\lambda})_{\lambda \in \Lambda}$ is $s$-localized.
    \item[(ii)] $\mathcal{U}$ and $\widetilde{\mathcal{U}}$ are mutually $s$-localized.
\end{itemize}
\end{theorem}

Another crucial property of $s$-localized g-frames is, that the series (\ref{gframerec}) associated with g-frame reconstruction not only converges in the underlying Hilbert space, but also in a whole family of associated (quasi-)Banach spaces, defined below. 


%
%
%
%
\begin{definition}\label{Hpwdef}\cite{köbaLOC25}
Let $s>d+r$ for some $r\geq 0$, let $\nu_r(x) = (1+\vert x \vert)^r$ denote the polynomial weight of order $r$ on $\mathbb{R}^d$, and assume that $m$ is a $\nu_r$-moderate weight. Further, let $\mathcal{U} = (U_{\lambda})_{\lambda \in \Lambda}$ be an $s$-localized g-frame for $\Hil$ with $s$-localized canonical dual g-frame $\widetilde{\mathcal{U}} =(\widetilde{U}_{\lambda})_{\lambda \in \Lambda}$. Then, for each $\frac{d}{s-r} < p < \infty$, the \emph{co-orbit space} $\Co^p_{m}(\mathcal{U})$ is defined as the completion of 
$$\Hil^{00}(\mathcal{U}):= D_{\mathcal{U}}(c_{00}(\Lambda;\Hil))$$
with respect to the (quasi-)norm $\Vert \, . \, \Vert_{\Co^p_{m }(\mathcal{U})}$ on $\Hil^{00}(\mathcal{U})$ given by
$$\Vert f \Vert_{\Co^p_{m }(\mathcal{U})} := \Vert C_{\widetilde{\mathcal{U}}} f \Vert_{\ell^p_{m}(\Lambda;\Hil)} = \left( \sum_{\lambda \in \Lambda} \Vert \widetilde{U}_{\lambda} f \Vert_{\Hil}^p   \right)^{\frac{1}{p}}.$$
\end{definition}

In the case $p=\infty$, the definition of the co-orbit space $\Co_{m}^{\infty}(\mathcal{U})$ is more technical. It can be defined as a weighted modification of a certain normable subspace of the topological completion of the Hausdorff locally convex vector space $(\Hil, \sigma(\Hil,\Hil^{00}(\widetilde{\mathcal{U}})))$ (compare with \cite{haubakö25}), or as the topological dual space of $\Co^1_{1/m}(\mathcal{U})$ \cite{köbaLOC25}. Here, we omit a rigorous discussion of the space $\Co_{m}^{\infty}(\mathcal{U})$ and directly refer to \cite{köbaLOC25} for the details. 

\begin{theorem}\label{isometrycor}\cite{köbaLOC25}
The space $\Co_{m}^{p}(\mathcal{U})$ is a Banach space for $1\leq p \leq \infty$ and a quasi-Banach space for $\frac{d}{s-r}<p<1$, and $C_{\widetilde{\mathcal{U}}}: \Co_{m}^{p}(\mathcal{U}) \longrightarrow \ell^p_{\omega}(\Lambda;\Hil)$ is an isometry for all $\frac{d}{s-r} < p \leq \infty$. Furthermore, for all $\frac{d}{s-r}<p \leq \infty$, 
$$\Co_{m}^{p}(\mathcal{U}) = \Co_{m}^{p}(\widetilde{\mathcal{U}})$$
with equivalent norms. In particular, 
$$\Vert (\widetilde{U}_{\lambda} f)_{\lambda \in \Lambda} \Vert_{\ell_m^p(\Lambda;\Hil)} \asymp \Vert (U_{\lambda} f)_{\lambda \in \Lambda} \Vert_{\ell_m^p(\Lambda;\Hil)}$$
are equivalent norms for $\Co_{m}^{p}(\mathcal{U})$.
\end{theorem}

Next, we consider the synthesis operator $D_{\mathcal{U}}: \ell_m^p(\Lambda, \Hil) \longrightarrow \Co_m^p(\mathcal{U})$. For simplicity, we use the same symbol for each of the operators 
$$D_{\mathcal{U}}:\ell_{\omega}^{p}(\Lambda;\Hil)\longrightarrow \Co^{p}_{m}(\mathcal{U}) \qquad (d/(s-r) < p < \infty),$$
densely defined on $c_{00}(\Lambda;\Hil)$ via 
$$D_{\mathcal{U}}(g_{\mu})_{\mu\in \Lambda} = \sum_{\mu\in \Lambda} U_{\mu}^* g_{\mu}.$$
Again, we omit the precise definition of  
$D_{\mathcal{U}}$ in the case $p=\infty$.  

\begin{theorem}\label{Donto}
\cite{köbaLOC25} Under the same assumptions as in Definition \ref{Hpwdef}, the synthesis operator $D_{\mathcal{U}}: \ell_{m}^{p}(\Lambda;\Hil) \longrightarrow  \Co_{m}^{p}(\mathcal{U})$ is bounded and surjective for each $\frac{d}{s-r}<p \leq \infty$ and 
$$\Vert f \Vert_{\Co_{m}^p(\mathcal{U})} \asymp \inf\left\lbrace \Vert (g_{\lambda})_{\lambda \in \Lambda} \Vert_{\ell^p_{m}(\Lambda;\Hil)} : (g_{\lambda})_{\lambda \in \Lambda}\in \ell^p_{\omega}(\Lambda;\Hil), f = D_{\mathcal{U}} (g_{\lambda})_{\lambda \in \Lambda} \right\rbrace.$$
\end{theorem}

Finally, we provide an alternative description of the spaces $\Co_m^p(\mathcal{U})$ for $p<\infty$.

\begin{proposition}\label{Vpw}
Under the same assumptions as in Definition \ref{slocalized}, $\Co_{m}^p(\mathcal{U})$ is the completion of the space 
$$V_{m}^p(\mathcal{U}) = \lbrace f\in \Hil : C_{\widetilde{\mathcal{U}}}f \in \ell^{p}_{m}(\Lambda;\Hil) \rbrace$$
with respect to $\Vert \, . \, \Vert_{\Co_{m}^p(\mathcal{U})}$ for each $\frac{d}{s-r}<p < \infty$. In particular, if $\ell^p_m(\Lambda,\Hil)$ is continuously embedded in $\ell^2(\Lambda,\Hil)$, then $\Co_{m}^p(\mathcal{U}) = V_{m}^p(\mathcal{U})$.
\end{proposition}

So far, the co-orbit spaces $\Co_{m}^p(\mathcal{U})$ seem to depend heavily on the underlying localized g-frame $\mathcal{U}$. However, if $\mathcal{W}$ is another localized g-frame and if $\mathcal{U}$ and $\mathcal{W}$ are mutually localized, then their associated co-orbit space coincide, see below.

\begin{theorem}\label{samecoorbitspaces}\cite{köbaLOC25}
Let $s>d+r$ for some $r\geq 0$ and assume that $m$ is a $\nu_r$-moderate weight. Assume that $\mathcal{U}$ and $\mathcal{W}$ both are intrinsically $s$-localized g-frames and assume that $\mathcal{U}$ and $\mathcal{W}$ are also mutually $s$-localized. Then, for all $\frac{d}{s-r} < p \leq \infty$, 
$$\Co_{m}^p(\mathcal{U}) = Co_{m}^p(\mathcal{W})$$
with equivalent norms.
\end{theorem}

%

\subsection{Time-frequency analysis and modulation spaces}

Modulation spaces arise naturally in the context of time-frequency analysis through the action of \emph{time-frequency shifts}, given by the projective unitary representation of the reduced Weyl--Heisenberg group on the Hilbert space $L^2(\mathbb{R}^d)$. For $z=(x,\omega)\in \mathbb{R}^{2d}$, the time-frequency shift $\pi(z)$ is defined as the composition of the translation operator $T_x:f(t)\mapsto f(t-x)$, and the modulation operator $M_{\omega}:f(t)\mapsto e^{2\pi i \omega t}f(t)$, that is, 
\begin{align*}
    \pi(z) = M_{\omega}T_x.
\end{align*}


\medskip

The \emph{Short-Time Fourier Transform (STFT)} of $f \in L^2(\mathbb{R}^d)$ with respect to a window $g\in L^2(\mathbb{R}^d)$ is defined by
\[
V_g f(z) := \langle f, \pi(z) g \rangle_{L^2},
\quad z\in \mathbb{R}^{2d}.
\]
Typical windows include compactly supported functions or rapidly decaying functions such as the normalized Gaussian $\varphi_0(t) = 2^{d/4} e^{-\pi t^2}$. For $f,g \in L^2(\mathbb{R}^d)$ the mapping $z \mapsto V_g f(z)$ belongs to $L^2(\mathbb{R}^{2d})$ and is uniformly continuous.
As a consequence of Moyal's identity $$\langle V_{g_1} f_1, V_{g_2} f_2 \rangle_{L^2(\mathbb{R}^{2d})}
= \langle f_1, f_2 \rangle_{L^2(\mathbb{R}^d)} \,
\overline{\langle g_1, g_2 \rangle_{L^2(\mathbb{R}^d)}},$$   $V_g: L^2(\mathbb{R}^d) \to L^2(\mathbb{R}^{2d})$ is an isometry and allows reconstruction by 
\[
f = \int_{\mathbb{R}^{2d}} V_g f(z)\, \pi(z) g \, dz . 
\]
We finally mention that its adjoint operator is 
\[
V_g^* F = \int_{\mathbb{R}^{2d}} F(z)\, \pi(z) g \, dz,
\]
with interpretation of integrals  in the weak sense. In other words,  
\(V_g^* V_g = I_{L^2(\mathbb{R}^d)}\).

The briefly introduced framework of the STFT now suggests the following definition of special cases of coorbit spaces, namely the {\em modulation spaces}. 

We begin by considering the weighted modulation space $M^1_{\nu_s}(\mathbb{R}^{d})$, where $\nu_s$ ($s\geq 0$) is the polynomial weight introduced in Subsection \ref{Weight functions}. The space $M^1_{\nu_s}(\mathbb{R}^{d})$ is defined as the collection of functions whose STFT with respect to the Gaussian window $\varphi_0(t) = 2^{\frac{d}{4}}e^{-\pi t^2}$ belongs to the weighted $L^1$ space, i.e., 
\[
M^1_v(\mathbb{R}^{d}) := \{ f \in L^2(\mathbb{R}^d) : V_{\varphi_0} f \in L^1_v(\mathbb{R}^{2d}) \}.
\]
For each $s\geq 0$, $M^1_{\nu_s}(\mathbb{R}^{d})$ is a Banach space with respect to the norm 
\[
\| f \|_{M^1_v} := \| V_{\varphi_0} f \|_{L^1_v}
\] and it is nontrivial, as it contains $\varphi_0$ itself. Moreover, the Schwartz space $\mathscr{S}(\mathbb{R}^d)$ is contained in $M^1_{\nu_s}(\mathbb{R}^{d})$ and $M^1_{\nu_s}(\mathbb{R}^{d})$ is invariant under time-frequency shifts. 
%
The unweighted case $M^1(\mathbb{R}^d)$ is known as \emph{Feichtinger's algebra}, a fundamental object in time-frequency analysis which provides an ideal space of test functions. We refer to Feichtinger's original paper \cite{feich81} and the survey \cite{jakob18} for further background.

\medskip

General modulation spaces are defined as follows. Given a $\nu_s$-moderate weight $m$ and exponents $1 \leq p,q \leq \infty$, one sets
\[
M^{p,q}_m(\mathbb{R}^{d}) := \{ f \in (M^1_v(\mathbb{R}^{d}))' : V_{\varphi_0} f \in L^{p,q}_m(\mathbb{R}^{2d}) \},
\]
with norm
\[
\| f \|_{M^{p,q}_m} := \| V_{\varphi_0} f \|_{L^{p,q}_m}.
\]
For any $g \in M^1_{\nu_s}(\mathbb{R}^{d})$, the space
\[
\{ f \in (M^1_v(\mathbb{R}^{d}))' : V_{g} f \in L^{p,q}_m(\mathbb{R}^{2d}) \}
\]
coincides with $M^{p,q}_m(\mathbb{R}^{d})$, and the corresponding norms are equivalent. For $p=q$ we abbreviate $M^{p}_m(\mathbb{R}^{d}):= M^{p,p}_m(\mathbb{R}^{d})$. Note that, as a direct consequence of the properties of the STFT with Gaussian window, we have
$M^{2}(\mathbb{R}^d) = L^2(\mathbb{R}^d)$. Finally, we collect the following result for later reference.

\begin{proposition}
\label{Mpprop}\cite{gr01}
If $1 \leq p_1 \leq p_2 \leq \infty$ and $m_2(z) \leq Cm_1(z)$ for some $C>0$, then $M^{p_1}_{m_1}(\mathbb{R}^d)$ is continuously and densely embedded in $M^{p_2}_{m_2}(\mathbb{R}^d)$.
\end{proposition}

\subsection{Operator coorbit spaces}

For two operators $\Phi,F \in \HS := \HS ( L^2(\mathbb{R}^d))$, the \emph{operator short-time Fourier transform} (\emph{operator STFT}) of $F$ with respect to a \emph{window} or \emph{atom} $\Phi$, $\mathcal{V}_{\Phi} F$, is defined by \cite{doelumcskr24} 
\begin{equation}
    \mathcal{V}_{\Phi} F (z) = \Phi^*\pi(z)^*F.
\end{equation}
The operator STFT thus defines an operator-valued function in phase space, and turns out to in fact define an isometry from $\mathcal{HS}$ to $L^2(\mathbb{R}^{2d};\mathcal{HS})$, for a fixed, normalized $\Phi\in\mathcal{HS}$. 

Let $\nu_s(z) = (1+\vert z \vert)^s$ be the polynomial weight of order $s\geq0$ on $\mathbb{R}^{2d}$ and $m$ be a $\nu_s$-moderate weight. Moreover, we set $\Phi_0 = \varphi_0 \otimes \varphi_0$, where $\varphi_0(t) = 2^{\frac{d}{4}}e^{-\pi t^2}$ denotes the $L^2$-normalized Gaussian. The space $\mathfrak{M}^1_{\nu_s}(\mathbb{R}^d)$ is defined by \cite{doelumcskr24}
\begin{align}\label{m1cond}
    \mathfrak{M}^1_{\nu_s}(\mathbb{R}^d) := \{F \in \mathcal{HS}: \mathcal{V}_{\Phi_0} F \in L^1_{\nu_s}(\mathbb{R}^{2d};\mathcal{HS}) \}
\end{align}
with corresponding norm 
$$\|F\|_{\mathfrak{M}^1_{\nu_s}(\mathbb{R}^d)} = \|\mathcal{V}_{\Phi_0} F\|_{L^1_{\nu_s}(\mathbb{R}^{2d};\mathcal{HS})},$$
and we denote its unweighted version, corresponding to $\nu_0 \equiv 1$, by $\mathfrak{M}^1(\mathbb{R}^d)$. To define more general \emph{operator coorbit spaces} $\mathfrak{M}^{p}_m(\mathbb{R}^d)$ for $1\leq p \leq \infty$, we need to resort to the space $\mathfrak{S}$ of Schwartz operators, cf.~\cite{KeylKiukasWerner2015}.\footnote{ 
The authors in that paper give the following intuitive characterization of Schwartz operators by rapid off-diagonal decay in the Hermite basis: 
Let $\{\psi_\alpha : \alpha \in \mathbb{N}_0^n\}$ be the $n$-dimensional Hermite 
function basis of $L^2(\mathbb{R}^n)$. 
An operator $T$ on $\mathcal{H}$ is a Schwartz operator 
if and only if for all $k \in \mathbb{N}$ we have
\[
\sup_{\alpha,\beta \in \mathbb{N}_0^n} 
\Big( (1+|\alpha|+|\beta|)^k \, 
\big| \langle \psi_\alpha, T \psi_\beta \rangle \big| \Big) < \infty.
\]}

The spaces $\mathfrak{M}^{p}_m(\mathbb{R}^d)$  are then defined by
\begin{align*}
    \mathfrak{M}^{p}_m(\mathbb{R}^d) := \{F \in \mathfrak{S}': \Phi_0^* \pi(z)^* F \in L^{p}_m(\mathbb{R}^{2d};\mathcal{HS}) \}.
\end{align*}
with norms 
$$\|F\|_{\mathfrak{M}^{p}_m} = \|\mathcal{V}_{\Phi_0} F\|_{L^{p}_m(\mathbb{R}^{2d};\HS)} = \left( \int_{\mathbb{R}^{2d}} \Vert \Phi_0^* \pi(z)^* F\Vert_{\HS}^p m(z)^p \, dz \right)^{\frac{1}{p}}$$
for $p<\infty$ and with the usual modification in the case $p=\infty$. 
As one might expect given the terminology, these spaces arising from the operator STFT act in many ways analogously to the modulation spaces arising from the (function) STFT. Recall the definition of the operator STFT $\mathcal{V}_{\Phi} F (z) = \Phi^*\pi(z)^*F$ which for a normalized Hilbert-Schmidt operator  $\Phi$ is an isometry and the associated operator Gabor space $\mathcal{V}_{\Phi}(\mathcal{HS})$ is a vector-valued reproducing kernel Hilbert space, and in \cite{doelumcskr24} an analog of the correspondence principle for the coorbit spaces $\mathfrak{M}^{p}_m(\mathbb{R}^d)$ like for Feichtinger's modulation spaces as well as an atomic decompositions for operators in $ \mathfrak{M}^{p}_m(\mathbb{R}^d)$. 


It can be shown that the spaces $\mathfrak{M}^p_m(\mathbb{R}^d)$ can be defined with respect to any window operator $\Phi\in\mathfrak{M}^1_{\nu_s}(\mathbb{R}^d)$, with equivalent norm. Furthermore, 
we have the continuous inclusions
\begin{align*}
    \mathfrak{M}^p_m(\mathbb{R}^d) \subset \mathfrak{M}^{p^\prime}_{\tilde{m}}(\mathbb{R}^d),
\end{align*}
for $1\leq p \leq p^\prime \leq \infty $ and $\tilde{m}(z) \lesssim m(z)$. 


\section{Identification of the spaces $\mathfrak{M}^p_m(\mathbb{R}^d)$}\label{Sec:Mpm}

In this section, we identify the operator coorbit spaces $\mathfrak{M}^p_m(\mathbb{R}^d)$ with the co-orbit spaces $\Co_m^p(\mathcal{G})$ associated with some $s$-localized g-frame $\mathcal{G} = (G_{\lambda})_{\lambda \in \Lambda}$ for the Hilbert space $\HS:= \HS(L^2(\mathbb{R}^d))$. This identification is interesting and useful for the following reasons. While the $\mathfrak{M}^p_m$-norms of some operator $F$ are given by the (continuous) $L^p_m(\mathbb{R}^{2d};\HS)$-norms of its operator STFT, said identification (Theorem \ref{mainchar}) guarantees norm equivalence with (discrete) $\ell^p_m(\Lambda;\HS)$-norms of the operator STFT of $F$ sampled on a suitable lattice $\Lambda \subset \mathbb{R}^{2d}$.  This  is much more convenient for both theoretical and practical purposes. In principal, this norm equivalence has already been shown in \cite{doelumcskr24}. However, by employing the machinery of $s$-localized g-frames, the definition of the co-orbit spaces $\Co_m^p(\mathcal{G})$ naturally extends to exponents $0<p<1$ in this case, yielding a canonical extension of the definition of $\mathfrak{M}^p_m(\mathbb{R}^d)$ to $0<p<1$. As we will see in Section \ref{Approximation theoretic results}, these resulting (quasi-)Banach spaces are relevant for sparse approximation of operators. Moreover, applying the results from Subsection \ref{Polynomially localized g-frames}, we obtain further characterizations of the operator coorbit spaces $\mathfrak{M}^p_m(\mathbb{R}^d)$ that are interesting in their own right.  

In the following we outline the construction of the co-orbit spaces $\Co_m^p(\mathcal{G})$ (which we will identify with the spaces $\mathfrak{M}^p_m(\mathbb{R}^d)$).

\noindent \textbf{Step 1.} We show that $(\Phi_0 \pi(\lambda)^*)_{\lambda \in \Lambda}$ is a g-frame for $L^2(\mathbb{R}^d)$. \footnote{It is straightforward to show that $(\Phi_0 \pi(\lambda)^*)_{\lambda \in \Lambda}$ is $s$-localized and that the co-orbit spaces associated with that g-frame are well-defined and non-trivial. However, the associated co-orbit space norms are given by $\ell^p_m(\Lambda; L^2(\mathbb{R}^{d}))$-norms instead of instead our desired $\ell^p_m(\Lambda;\HS(L^2(\mathbb{R}^{d})))$-norms.  Thus, $(\Phi_0 \pi(\lambda)^*)_{\lambda \in \Lambda} \subset \B(L^2(\mathbb{R}^d))$ is a g-frame for the "wrong" Hilbert space $L^2(\mathbb{R}^d)$ and we need to \emph{lift} it to a g-frame for the Hilbert space $\HS(L^2(\mathbb{R}^{d}))$, that is, a g-frame consisting of operators from $\B(\HS(L^2(\mathbb{R}^d)))$.}

\noindent \textbf{Step 2.} We \emph{lift} the g-frame $(\Phi_0 \pi(\lambda)^*)_{\lambda \in \Lambda}$ to a g-frame $\mathcal{G} = (G_{\lambda})_{\lambda \in \Lambda}$ for the Hilbert space $\HS(L^2(\mathbb{R}^{d}))$.

\noindent \textbf{Step 3.} We show that the g-frame $\mathcal{G} = (G_{\lambda})_{\lambda \in \Lambda}$ for $\HS(L^2(\mathbb{R}^{d}))$ is $s$-localized for all $s\geq 0$. 

\noindent \textbf{Step 4.} Having introduced the $s$-localized g-frame $\mathcal{G}$, we may apply the machinery presented in Subsection \ref{Polynomially localized g-frames} and obtain the family of associated co-orbit spaces $\Co_m^p(\mathcal{G})$ for \emph{all} values of $p\in (0,\infty]$. 

\

We start with the lifting property outlined in the Step 2. This lifting trick has already been performed in \cite[Proposition 6.12]{doelumcskr24} but also holds true in abstract Hilbert space $\Hil$ and is perfectly compatible with g-frame duality. In order to formulate the result, the following notation is required. 

Let $\Hil$ be a Hilbert space and $W\in \B(\Hil)$ be a bounded operator on $\Hil$. Then $L_W \in \B(\HS(\Hil))$ denotes the left-composition operator defined by 
\begin{equation}\label{leftmult}
L_W H := WH.    
\end{equation}
Obviously, $(L_W)^* = L_{W^*}$ and $L_{W_1}L_{W_2} = L_{W_1 W_2}$ for $W_1, W_2 \in \B(\Hil)$. If $W$ is invertible in $\B(\Hil)$, then $L_W$ is invertible in $\B(\HS(\Hil))$ with inverse given by $(L_W)^{-1} = L_{W^{-1}}$.

\begin{proposition}[Lifting Property]\label{HilversusHS}
Let $\Hil$ be a separable Hilbert space and assume that  $(W_{\lambda})_{\lambda \in \Lambda}$ and $(U_{\lambda})_{\lambda \in \Lambda}$ are countable families contained in $\HS(\Hil)$. Then the following hold.
\begin{itemize}
    \item[(i)] $(U_{\lambda})_{\lambda \in \Lambda}$ is a g-frame for $\Hil$ with g-frame bounds $A\leq B$ if and only if $(L_{U_{\lambda}})_{\lambda \in \Lambda}$ is a g-frame for $\HS(\Hil)$ with g-frame bounds $A\leq B$.
    \item[(ii)] $(W_{\lambda})_{\lambda \in \Lambda}$ and $(U_{\lambda})_{\lambda \in \Lambda}$ is a pair of dual g-frames in $\Hil$ if and only if $(L_{W_{\lambda}})_{\lambda \in \Lambda}$ and $(L_{U_{\lambda}})_{\lambda \in \Lambda}$ is a pair of dual g-frames in $\HS(\Hil)$.
    \item[(iii)] If (i) holds true, then the canonical dual g-frame of $(L_{U_{\lambda}})_{\lambda \in \Lambda}$ in $\HS(\Hil)$ is given by $(L_{\widetilde{U}_{\lambda}})_{\lambda \in \Lambda}$, where $(\widetilde{U}_{\lambda})_{\lambda \in \Lambda}$ denotes the canonical dual g-frame of $(U_{\lambda})_{\lambda \in \Lambda}$ in $\Hil$.
    \item[(iv)] In case $\Hil = L^2(\mathbb{R}^d)$ and assuming that $(U_{\lambda})_{\lambda \in \Lambda} = (\Phi^* \pi(\lambda)^*)_{\lambda \in \Lambda}$ is a g-Bessel sequence, where $\Phi \in \HS(L^2(\mathbb{R}^d))$, then $S_{\mathcal{U}}$ commutes with $\pi(\mu)$ for every $\mu\in \Lambda$, and likewise, $S_{L_{\mathcal{U}}}$ commutes with $L_{\pi(\mu)}$ for every $\mu\in \Lambda$. In particular, if both $(U_{\lambda})_{\lambda \in \Lambda}$ and $(L_{U_{\lambda}})_{\lambda \in \Lambda}$ are g-frames for their respective spaces, then $S_{\mathcal{U}}^{-1}$ commutes with $\pi(\mu)$ for every $\mu\in \Lambda$, and likewise, $S_{L_{\mathcal{U}}}^{-1}$ commutes with $L_{\pi(\mu)}$ for every $\mu\in \Lambda$.
\end{itemize}
   
\end{proposition}

\begin{proof}
The proof of (i) completely analogous to the proof of \cite[Proposition 6.12]{doelumcskr24}.

(ii) First assume that $(W_{\lambda})_{\lambda \in \Lambda}$ is a dual g-frame of $(U_{\lambda})_{\lambda \in \Lambda}$ in $\Hil$. Then $f=\sum_{\lambda \in \Lambda}W_{\lambda}^* U_{\lambda} f$ for all $f\in \Hil$ and thus $Hf=\sum_{\lambda \in \Lambda}W_{\lambda}^* U_{\lambda} Hf$ for all $H\in \HS(\Hil)$ and all $f\in \Hil$. Let $(e_n)_{n\in \mathbb{N}}$ be any orthonormal basis for $\Hil$. Then we obtain 
$$\left\Vert H - \sum_{\lambda \in \Lambda} L_{W_{\lambda}}^* L_{U_{\lambda}} H \right\Vert_{\HS(\Hil)}^2 = \sum_{n\in \mathbb{N}} \left\Vert He_n - \sum_{\lambda \in \Lambda}W_{\lambda}^* U_{\lambda} H e_n \right\Vert_{\Hil}^2 = 0.$$
Since both $(W_{\lambda})_{\lambda \in \Lambda}$ and $(U_{\lambda})_{\lambda \in \Lambda}$ are g-frames for $\Hil$ by assumption, (i) implies that $(L_{W_{\lambda}})_{\lambda \in \Lambda}$ and $(L_{U_{\lambda}})_{\lambda \in \Lambda}$ are g-frames for $\HS(\Hil)$. At the same time, the latter computation shows that they form a pair of dual g-frames in $\HS(\Hil)$.

The converse statement is proven by specifying to Hilbert-Schmidt operators of the form $f\otimes e_k$ ($f\in \Hil$) and similar reasoning.

(iii) Assume that $\mathcal{U} = (U_{\lambda})_{\lambda \in \Lambda}$ is a g-frame for $\Hil$ and that $L_{\mathcal{U}} = (L_{U_{\lambda}})_{\lambda \in \Lambda}$ is a g-frame for $\HS(\Hil)$. Then, a direct calculation shows that their respective g-frame operators satisfy the relation $S_{L_{\mathcal{U}}} = L_{S_{\mathcal{U}}}$. This implies $(S_{L_{\mathcal{U}}})^{-1} = (L_{S_{\mathcal{U}}})^{-1} = L_{S_{\mathcal{U}}^{-1}}$ and the claim follows.

(iv) By assumption, the g-frame operator $S_{\mathcal{U}}$ is bounded. Since, for all $\mu,\lambda \in \Lambda$, $\pi(\mu)\pi(\lambda) = c(\mu,\lambda) \pi(\mu + \lambda)$ ($c(\mu,\lambda)$ some phase factor) we directly compute
$$\pi(\mu)S_{\mathcal{U}}\pi(\mu)^* = \sum_{\lambda \in \Lambda} \pi(\mu)\pi(\lambda) \Phi \Phi^*\pi(\lambda)^* \pi(\mu)^* = \sum_{\lambda \in \Lambda} \pi(\mu+\lambda) \Phi \Phi^*\pi(\mu+\lambda)^* = S_{\mathcal{U}}.$$
Thus $\pi(\mu)$ and $S_{\mathcal{U}}$ commute. The second statement follows immediately from the first statement, since $S_{L_{\mathcal{U}}} = L_{S_{\mathcal{U}}}$. The last statement follows readily.
\end{proof}

Now we apply the above result to the Hilbert spaces $\Hil = L^2(\mathbb{R}^d)$ and $\HS := \HS(L^2(\mathbb{R}^d))$. Recall that $\Phi_0 = \varphi_0 \otimes \varphi_0$, where $\varphi_0(t) = 2^{\frac{d}{4}}e^{-\pi t^2} \in L^2(\mathbb{R}^d)$ denotes the $L^2$-normalized Gaussian. 

\begin{theorem}\label{polydecayprop}
There exists a full-rank lattice $\Lambda \subset \mathbb{R}^{2d}$ such that the family 
$$(G_{\lambda})_{\lambda \in \Lambda}, \qquad G_{\lambda} := L_{\Phi_0\pi(\lambda)^*} \quad (\lambda \in \Lambda)$$
is a g-frame for $\HS(L^2(\mathbb{R}^d))$. Furthermore, for each $s \geq 0$, there exists some constant $C = C(s)>0$, such that     
\begin{equation}\label{polydecay}
\Vert G_{\lambda} G_{\mu}^*\Vert_{\B(\HS(L^2(\mathbb{R}^{2d})))} \leq C(1+\vert \lambda - \mu\vert)^{-s} \qquad (\forall \lambda , \mu\in \Lambda).
\end{equation}  
\end{theorem}

\begin{proof}
According to \cite[Corollary 6.19]{skrett21} there exists a natural number $N$ such that the family $(\pi(\lambda)\Phi_0\pi(\lambda)^*)_{\lambda \in \Lambda}$ is a g-frame for the Hilbert space $L^2(\mathbb{R}^d)$, where $\Lambda = \frac{1}{N}\mathbb{Z}^{d}\times \frac{1}{N}\mathbb{Z}^{d}$. Since each $\pi(\lambda)$ is unitary, the latter is equivalent to $(\Phi_0\pi(\lambda)^*)_{\lambda \in \Lambda}$ being a g-frame for $L^2(\mathbb{R}^d)$. By the lifting property (Proposition \ref{HilversusHS} (i)) the latter is equivalent to $(G_{\lambda})_{\lambda \in \Lambda}$ being a g-frame for $\HS(L^2(\mathbb{R}^d))$. It remains to show (\ref{polydecay}). To this end, recall that $\varphi_0 \in S(\mathbb{R}^d)$ is a Schwartz function. Since $S(\mathbb{R}^d)=\bigcap_{s\geq0}M_{\nu_s}^{\infty}(\mathbb{R}^d)$ (see \cite[Proposition 11.3.1]{grochenigtfa}) we particularly have that $\varphi_0 \in M_{\nu_s}^{\infty}(\mathbb{R}^d)$ for every $s\geq0$. Now, if $(e_n)_{n\in \mathbb{N}}$ denotes any orthonormal basis for $L^2(\mathbb{R}^d)$, we may estimate
\begin{flalign}
\Vert G_{\lambda} G_{\mu}^*\Vert_{\B(\HS)} &= \Vert L_{\Phi_0\pi(\lambda)^*\pi(\mu)\Phi_0}\Vert_{\B(\HS)} \notag \\
&\leq \Vert \Phi_0\pi(\lambda)^*\pi(\mu)\Phi_0\Vert_{\HS} \notag \\
&= \left( \sum_{n\in \mathbb{N}} \Big\Vert \langle e_n,\varphi_0 \rangle_{L^2} \langle \pi(\mu)\varphi_0, \pi(\lambda)\varphi_0 \rangle_{L^2} \varphi_0 \Big\Vert_{L^2}^2 \right)^{\frac{1}{2}} \notag \\
&= \vert \langle \pi(\mu)\varphi_0, \pi(\lambda)\varphi_0 \rangle_{L^2} \vert \notag \\
&= \vert V_{\varphi_0} \varphi_0(\lambda - \mu)\vert \notag \\
&\leq \Vert \varphi_0\Vert_{M_{\nu_s}^{\infty}(\mathbb{R}^d)}(1+\vert \lambda - \mu\vert)^{-s},\notag
\end{flalign}
where we used the covariance property of the STFT \cite{grochenigtfa} in the fifth line.
\end{proof}

\begin{remark}\label{framesets}
The existence of the lattice $\Lambda$ in the above proof is non-constructive. As already noted in the proof of Proposition \ref{HilversusHS} (iv), $(L_{\Phi_0\pi(\lambda)^*})_{\lambda \in \Lambda}$ being a g-frame for $\HS(L^2(\mathbb{R}^d))$ is equivalent to the Gabor system $(\pi(\lambda)\varphi_0)_{\lambda\in \Lambda}$ being a frame for $L^2(\mathbb{R}^d)$. Even for a window as "nice" as the Gaussian $\varphi_0$, characterizing the \emph{frame set}, i.e., the set of all $a,b >0$ such that $(\pi(\lambda)\varphi_0)_{\lambda \in a\mathbb{Z}^d \times b\mathbb{Z}^d}$ is a frame, is a challenging problem and an active field of research until today \cite{ja96,jastro02,ja03,gröstö13,GR14,fashazlo25}. In the case $d=1$, this problem has been solved by Lyubarksi and Seip-Wallsten in \cite{Lyu92,sei92,seiwal92}, where it is shown, that  $(\pi(\lambda)\varphi_0)_{\lambda \in a\mathbb{Z} \times b\mathbb{Z}}$ is a frame for $L^2(\mathbb{R})$ if and only if $ab<1$. Already in the case $d=2$ the situation becomes more complicated and \cite{luxu23} studies a class of non-rational lattices, so-called transcendental lattices. 
\end{remark}

Note that (\ref{polydecay}) says that the g-frame $\mathcal{G} = (G_{\lambda})_{\lambda \in \Lambda}$ is $s$-localized for \emph{every} $s>0$ and particularly for \emph{every} $s>2d$. Since the lattice $\Lambda = \frac{1}{N}\mathbb{Z}^{d}\times \frac{1}{N}\mathbb{Z}^{d} \subset \mathbb{R}^{2d}$ occurring in Theorem \ref{polydecayprop} is a relatively separated set, we may employ the machinery presented in Subsection \ref{Some background on g-frames and localization}. In particular, for each $\nu_r$-moderate weight $m$ ($r\geq 0$ arbitrary), the co-orbit space 
$\Co_m^p(\mathcal{G})$ is a well-defined Banach space for each $p\in [1,\infty]$ and a well-defined quasi-Banach space for each $0<p<1$. 

The g-frame related operators associated with the g-frame $\mathcal{G} = (G_\lambda)_{\lambda\in \Lambda} = (L_{\Phi_0\pi(\lambda)^*})_{\lambda\in \Lambda}$ from above have been studied (in a slightly more general setting) in \cite{doelumcskr24}. Before we proceed, we collect the results from \cite{doelumcskr24} we will need. 

\begin{proposition}\label{analbound}\cite{doelumcskr24}
Let $\Phi\in \mathfrak{M}^{1}_{\nu_s}(\mathbb{R}^d)$ and $1\leq p \leq \infty$ be arbitrary. Then the analysis operator $C_{\Phi}: \mathfrak{M}^{p}_m(\mathbb{R}^d)\longrightarrow \ell^{p}_{m}(\Lambda,\mathcal{HS})$, defined by 
    \begin{align*}
        C_{\Phi}(F) = (\Phi^*\pi(\lambda)^*F)_{\lambda\in\Lambda},
    \end{align*}
    is well-defined and bounded by a constant, that depends only on $\Lambda$ and $s$.
\end{proposition}

\begin{proposition}\label{synthbound}\cite{doelumcskr24}
Let $\Phi\in \mathfrak{M}^{1}_{\nu_s}(\mathbb{R}^d)$ and $1\leq p \leq \infty$ be arbitrary. Then the synthesis operator $D_{\Phi}: \ell^{p}_{m}(\Lambda;\mathcal{HS}) \longrightarrow \mathfrak{M}^{p}_{{m}}(\mathbb{R}^d)$, defined by 
    \begin{align*}D_{\Phi}((G_{\lambda})_{\lambda\in\Lambda}) = \sum_{\lambda\in\Lambda} \pi(\lambda)\Phi G_{\lambda},
    \end{align*}
is well-defined and bounded by a constant, that depends only on $\Lambda$ and $s$. Moreover, the above series converges unconditionally in $\mathfrak{M}^{p}_{{m}}(\mathbb{R}^d)$ for $1\leq p < \infty$, and unconditionally weak* in the case $p=\infty$.
\end{proposition}

\begin{proposition}\label{framerecMp}\cite{doelumcskr24}
Let $\Phi,\Psi \in\mathfrak{M}^1_{\nu_s}(\mathbb{R}^d)$. If $(L_{\Psi \pi(\lambda)^*})_{\lambda \in \Lambda}$ and $(L_{\Phi \pi(\lambda)^*})_{\lambda \in \Lambda}$ are dual g-frames in $\HS$, then $D_{\Psi}C_{\Phi} = \mathcal{I}_{\B(\mathcal{HS})}$ and $D_{\Psi} C_{\Phi} = \mathcal{I}_{\B(\mathfrak{M}^{p}_m)}$ for all $1\leq p\leq \infty$, where the associated series converges unconditionally in $\mathfrak{M}^{p}_{{m}}(\mathbb{R}^d)$ for $1\leq p < \infty$, and unconditionally weak* in the case $p=\infty$.
\end{proposition}

We are now prepared to prove the main result of this section. It can be seen as an operator analogue of \cite[Theorem 4.1]{forngroech05}.

\begin{theorem}\label{mainchar}
Let $\mathcal{G} = (G_{\lambda})_{\lambda \in \Lambda} = (L_{\pi(\lambda)\Phi_0\pi(\lambda)^*})_{\lambda \in \Lambda}$ be the g-frame occurring in Theorem \ref{polydecayprop}, $m$ be a $\nu_r$-moderate weight for some $r\geq 0$ and $1\leq p \leq \infty$ be arbitrary. Then
$$\mathfrak{M}_m^p(\mathbb{R}^d) = \Co_m^p(\mathcal{G})$$
with equivalent norms.
\end{theorem}

\begin{proof}
The case $p=\infty$ can be deduced from the case $p=1$ via the duality relations $(\mathfrak{M}_{1/m}^1(\mathbb{R}^d))' \cong \mathfrak{M}_m^{\infty}(\mathbb{R}^d)$ \cite[Proposition 5.17]{doelumcskr24} and $(\Co_{1/m}^1(\mathcal{G}))' \simeq \Co_m^{\infty}(\mathcal{G})$ \cite[Theorem 3.29]{köbaLOC25} after noticing that the weight $1/m$ is $\nu_r$-moderate (\cite[Lemma 11.1.1]{gr01}). 

We fix some $1\leq p < \infty$ and let $F\in \Co_m^p(\mathcal{G})$ be arbitrary. By Theorem \ref{Donto} the g-frame synthesis operator $D_{\mathcal{G}}: \ell_{m}^{p}(\Lambda;\HS) \longrightarrow  \Co_{m}^{p}(\mathcal{G})$ is onto. Thus, there exists some $(H_{\lambda})_{\lambda \in \Lambda}\in \ell_{m}^{p}(\Lambda;\HS)$ such that $F = \sum_{\lambda \in \Lambda} \pi(\lambda)\Phi_0 H_{\lambda}$. 
At the same time, 
Proposition \ref{synthbound} says that synthesis operator $D_{\Phi_0}$ associated to the operator window $\Phi_0$ is bounded from $\ell_{m}^{p}(\Lambda;\HS)$ into $\mathfrak{M}_{m}^{p}(\mathbb{R}^d)$. Thus $F = D_{\Phi_0} (H_{\lambda})_{\lambda \in \Lambda}$ is contained in $\mathfrak{M}_{m}^{p}(\mathbb{R}^d)$ and we may estimate 
$$\Vert F \Vert_{\mathfrak{M}_{m}^{p}} = \left\Vert \sum_{\lambda \in \Lambda} \pi(\lambda)\Phi_0 H_{\lambda} \right\Vert_{\mathfrak{M}_{m}^{p}} \leq C \Vert (H_{\lambda})_{\lambda \in \Lambda}\Vert_{\ell_{m}^{p}(\Lambda;\HS)}.$$
The latter is true for all $(H_{\lambda})_{\lambda \in \Lambda}\in \ell_{m}^{p}(\Lambda;\HS)$ with $D_{\mathcal{G}} (H_{\lambda})_{\lambda \in \Lambda} = F$. Therefore, 
\begin{flalign}
\Vert F \Vert_{\mathfrak{M}^p_{m}} &\leq C \inf\left\lbrace \Vert (H_{\lambda})_{\lambda \in \Lambda} \Vert_{\ell_{m}^{p}(\Lambda;\HS)} : (H_{\lambda})_{\lambda \in \Lambda}\in \ell_{m}^{p}(\Lambda;\HS), F = D_{\mathcal{G}} (H_{\lambda})_{\lambda \in \Lambda} \right\rbrace \notag \\
&\leq C'\Vert F\Vert_{\Co_m^p(\mathcal{G})},\notag
\end{flalign}
where we used the norm equivalence stated in Theorem \ref{Donto} in the last estimate.

Conversely, let $F\in \mathfrak{M}_{m}^p(\mathbb{R}^d)$ be arbitrary. According to Proposition \ref{analbound}, $C_{\Phi_0}F = (\Phi_0\pi(\lambda)^*F)_{\lambda \in \Lambda} \in \ell_{m}^{p}(\Lambda;\HS)$. 
Next, let $S_{0}$ denote the frame operator associated to the g-frame $(\Phi_0\pi(\lambda)^*)_{\lambda \in \Lambda}$ for $L^2(\mathbb{R}^d)$ appearing in the proof of Theorem \ref{polydecayprop}. Then, according to \cite[Subsection 5.2.1]{skrett21}, the canonical dual g-frame of $(\Phi_0\pi(\lambda)^*)_{\lambda \in \Lambda}$ in $L^2(\mathbb{R}^d)$ is given by $(\Phi_0 S_{0}^{-1}\pi(\lambda)^*)_{\lambda \in \Lambda}$. By Proposition \ref{HilversusHS}, $(L_{\Phi_0 S_{0}^{-1}\pi(\lambda)^*})_{\lambda \in \Lambda}$ equals the canonical dual g-frame $(\widetilde{G}_{\lambda})_{\lambda \in \Lambda}$ of $(G_{\lambda})_{\lambda \in \Lambda}$ in $\HS$. Furthermore, we also note that $\Phi_0 S_{0}^{-1} \in \mathfrak{M}^1_{\nu_s}(\mathbb{R}^d)$ for each $s\geq 0$, since
\begin{flalign}
\int_{\mathbb{R}^{2d}} \Vert \Phi_0 \pi(z)^* \Phi_0 S_{0}^{-1}\Vert_{\HS}\nu_s(z) \, dz &\leq \Vert S_{0}^{-1}\Vert_{\B(L^2(\mathbb{R}^d))} \int_{\mathbb{R}^{2d}} \Vert \Phi_0 \pi(z)^* \Phi_0 \Vert_{\HS}\nu_s(z) \, dz \notag \\
&= \Vert S_{0}^{-1}\Vert_{\B(L^2(\mathbb{R}^d))} \Vert \Phi_0\Vert_{\mathfrak{M}^1_{\nu_s}(\mathbb{R}^d)} \notag \\
&= \Vert S_{0}^{-1}\Vert_{\B(L^2(\mathbb{R}^d))} \Vert \varphi_0\Vert_{M^1_{\nu_s}(\mathbb{R}^d)} <\infty \notag
\end{flalign}
for each $s\geq 0$. Now, substituting $\Phi = \Phi_0$ and $\Psi = \Phi_0 S_{0}^{-1}$ in Proposition \ref{framerecMp} yields that 
\begin{flalign}
F &= D_{\Phi_0 S_{0}^{-1}} C_{\Phi_0}F \notag \\
&= \sum_{\lambda \in \Lambda} \pi(\lambda) S_{0}^{-1} \Phi_0^2\pi(\lambda)^*F \notag \\ 
&= D_{\widetilde{\mathcal{G}}}(\Phi_0\pi(\lambda)^*F)_{\lambda \in \Lambda}. \notag 
\end{flalign}
Thus $F\in \Co_m^p(\mathcal{G})$ according to Theorem \ref{Donto}.

Altogether, we have shown that $\Co_{m}^{p}(\mathcal{G}) = \mathfrak{M}_m^p(\mathbb{R}^d)$ and that the identity map $id:\Co_{m}^{p}(\mathcal{G}) \longrightarrow \mathfrak{M}_m^p(\mathbb{R}^d)$ is continuous. Therefore, the norms on $\Co_{m}^{p}(\mathcal{G})$ and $\mathfrak{M}_m^p(\mathbb{R}^d)$ are equivalent as a consequence of the bounded inverse theorem. 
\end{proof}

Recall that the spaces $\Co_{m}^{p}(\mathcal{G})$ are also well-defined quasi-Banach spaces for all $0<p<1$ (where $m$ is a $\nu_r$-moderate weight for any $r\geq 0$), since the g-frame $(G_{\lambda})_{\lambda \in \Lambda}$ from Theorem \ref{polydecayprop} is $s$-localized for \emph{all} $s>2d$. Thus, the above characterization provides a meaningful way to extend the definition of the spaces $\mathfrak{M}_m^p(\mathbb{R}^d)$ to exponents $0<p<1$.

\begin{definition}[Operator coorbit spaces for $0<p<1$]\label{Mpdef}
Let $m$ be any $\nu_r$-moderate weight, where $r\geq 0$, let $0<p<1$ and let $\mathcal{G} =(G_{\lambda})_{\lambda \in \Lambda}$ be the g-frame from Theorem \ref{polydecay}. Then we define
$$\mathfrak{M}_m^p(\mathbb{R}^d) := \Co_{m}^{p}(\mathcal{G}).$$
\end{definition}

It is worthwhile mentioning the following consequence of Theorem \ref{mainchar}, Theorem \ref{isometrycor} and Proposition \ref{Vpw}.

\begin{corollary}\label{Mpeasy}
For $0<p\leq 2$ it holds
$$\mathfrak{M}^p(\mathbb{R}^d) = \left\lbrace F\in \HS(L^2(\mathbb{R}^d)): (G_{\lambda} F)_{\lambda \in \Lambda} = (\mathcal{V}_{\Phi_0} F (\lambda) )_{\lambda \in \Lambda} \in \ell^p(\Lambda;\HS)  \right\rbrace$$
and 
$$\Vert F \Vert_{\mathfrak{M}^p(\mathbb{R}^d)} \asymp \Vert (G_{\lambda} F)_{\lambda \in \Lambda} \Vert_{\ell^p(\Lambda;\HS)} = \Vert (\mathcal{V}_{\Phi_0} F (\lambda))_{\lambda \in \Lambda} \Vert_{\ell^p(\Lambda;\HS)}.$$
\end{corollary}

Theorem \ref{samecoorbitspaces} already hints that the g-frame $\mathcal{G}$ (respectively the Gaussian state $\Phi_0$) may be replaced by other g-frames $\mathcal{W}$ (respectively other operator windows $\Psi$) in Theorem \ref{mainchar}. This is indeed the case, as the following result shows, yielding further options of operator dictionaries for the sparse approximation of operators, see Section \ref{Approximation theoretic results}. However, be wary that here the g-frame condition is an assumption (c.f. Remark \ref{framesets}).  

\begin{theorem}\label{mainchar2}
Let $s>2d+r$ for some $r\geq 0$ and $\Psi$ be an operator of the form
$$\Psi= \sum_{m=1}^{\infty} \phi_m \otimes \psi_m ,$$
where $\phi_m \in M_{\nu_s}^{1}(\mathbb{R}^d)$ and $\psi_m \in L^2(\mathbb{R}^d)$ for all $m\in \mathbb{N}$ and
$$\sum_{m=1}^{\infty} \Vert \phi_m \Vert_{M_{\nu_s}^{1}} \Vert \psi_m \Vert_{L^{2}} <\infty .$$
Assume that $(\Psi^* \pi(\lambda)^*)_{\lambda \in \Lambda}$ is a g-frame for $L^2(\mathbb{R}^d)$, where $\Lambda$ is a lattice as in Theorem \ref{polydecayprop}. Then 
$$\mathcal{W} = (W_{\lambda})_{\lambda \in \Lambda}, \qquad W_{\lambda} := L_{\Psi^* \pi(\lambda)^*} \quad (\lambda \in \Lambda)$$
is a g-frame for $\HS(L^2(\mathbb{R}^d))$ which is intrinsically $s$-localized. Furthermore, $\mathcal{G}$ and $\mathcal{W}$ are mutually $s$-localized (where $\mathcal{G}$ denotes the g-frame from Theorem \ref{polydecayprop}), and for all $\nu_r$-moderate weights $m$ and all $\frac{2d}{s-r}<p\leq \infty$ it holds
$$\mathfrak{M}_m^p(\mathbb{R}^d) = \Co_m^p(\mathcal{W})$$
with equivalent norms.
\end{theorem}

\begin{proof}
Let $\Psi$ and $\Lambda$ be as above. Let $C_1$ be the constant associated with the continuous embedding $M^1_{\nu_s}(\mathbb{R}^d) \hookrightarrow L^2(\mathbb{R}^d)$ and $C_2$ be the constant associated with the continuous embedding $M^1_{\nu_s}(\mathbb{R}^d) \hookrightarrow M^{\infty}_{\nu_s}(\mathbb{R}^d)$ (see Proposition \ref{Mpprop}). Upon a rescaling argument we may assume W.L.O.G. that $\Vert \psi_m \Vert_{L^2} = 1$ for all $m\in \mathbb{N}$ so that $C_3 := \sum_{m=1}^{\infty} \Vert \phi_m \Vert_{M_{\nu_s}^{1}} < \infty$. 

\noindent \textbf{Step 1.} We first prove that $\Psi$ is a Hilbert-Schmidt operator on $L^2(\mathbb{R}^d)$. Indeed, let $(e_n)_{n=1}^{\infty}$ be an orthonormal basis for $L^2(\mathbb{R}^d)$. Then 
\begin{flalign}
\Vert \Psi \Vert_{\HS(L^2)}^2 &= \sum_{n=1}^{\infty} \left\Vert \sum_{m=1}^{\infty} \langle e_n, \psi_m \rangle \phi_m  \right\Vert_{L^2}^2 \notag \\
&= \sum_{n=1}^{\infty} \sum_{m=1}^{\infty} \sum_{m'=1}^{\infty} \langle e_n, \psi_m \rangle \langle \psi_{m'}, e_n \rangle \langle \phi_m , \phi_{m'} \rangle \notag \\
&= \sum_{m=1}^{\infty} \sum_{m'=1}^{\infty} \sum_{n=1}^{\infty}\langle e_n, \psi_m \rangle \langle \psi_{m'}, e_n \rangle \langle \phi_m , \phi_{m'} \rangle \notag \\
&= \sum_{m=1}^{\infty} \sum_{m'=1}^{\infty} \langle \psi_{m'}, \psi_m \rangle \langle \phi_m , \phi_{m'} \rangle \notag \\
&\leq C_1^2\sum_{m=1}^{\infty} \sum_{m'=1}^{\infty} \Vert \phi_{m'} \Vert_{M^1_{\nu_s}} \Vert \phi_m \Vert_{M^1_{\nu_s}} \Vert \psi_{m'} \Vert_{L^2} \Vert \psi_m \Vert_{L^2} \notag \\
&=  C_1^2 C_3^2, \notag
\end{flalign}
where the order of summation may be changed due to absolute convergence of the series 
\begin{flalign}
&\sum_{m=1}^{\infty} \sum_{m'=1}^{\infty} \sum_{n=1}^{\infty} \vert \langle e_n, \psi_m \rangle \vert \vert \langle \psi_{m'}, e_n \rangle \vert \vert \langle \phi_m , \phi_{m'} \rangle \vert \notag \\
&\leq \sum_{m=1}^{\infty} \sum_{m'=1}^{\infty} \vert \langle \phi_m , \phi_{m'} \rangle \vert \left( \sum_{n=1}^{\infty} \vert \langle e_n, \psi_m \rangle \vert^2 \right)^{\frac{1}{2}} \left( \sum_{n=1}^{\infty} \vert \langle \psi_{m'}, e_n \rangle \vert^2 \right)^{\frac{1}{2}} \notag \\
&= \sum_{m=1}^{\infty} \sum_{m'=1}^{\infty} \vert \langle \phi_m , \phi_{m'} \rangle \vert \leq C_1^2 C_3^2 . \notag
\end{flalign}

\noindent \textbf{Step 2.} We show that $\mathcal{W}$ is a g-frame for $\HS(L^2(\mathbb{R}^d))$. Since $\Psi \in \HS(L^2(\mathbb{R}^d))$ according to Step 1, we also have $\Psi^*\pi(\lambda)^* \in \HS(L^2(\mathbb{R}^d))$ for each $\lambda \in \Lambda$. Therefore, since $(\Psi^* \pi(\lambda)^*)_{\lambda \in \Lambda}$ is a g-frame for $L^2(\mathbb{R}^d)$ by assumption, we may conclude from the lifting property (Proposition \ref{HilversusHS}) that $(L_{\Psi^* \pi(\lambda)^*})_{\lambda \in \Lambda}$ is a g-frame for $\HS(L^2(\mathbb{R}^d))$.

\noindent \textbf{Step 3.}
We show that
$$\vert V_{\phi_k} \phi_m (\lambda - \mu) \vert \leq  \frac{C_2}{\nu_s(\lambda - \mu)} \Vert \phi_m \Vert_{M^{1}_{\nu_s}} \Vert \phi_k \Vert_{M^1_{\nu_s}} \qquad (\forall k,m \in \mathbb{N}).$$
Indeed, since $\Vert \varphi_0 \Vert_{L^2}=1$, the STFT satisfies the pointwise estimate 
$$\vert V_{\phi_k} \phi_m \vert \leq \frac{1}{\Vert \varphi_0 \Vert_{L^2}^2} \vert V_{\varphi_0} \phi_m \vert \ast \vert V_{\phi_k} \varphi_0 \vert = \vert V_{\varphi_0} \phi_m \vert \ast \vert V_{\phi_k} \varphi_0 \vert$$ 
(see \cite[Theorem 11.3.7]{gr01}). Applying Young's inequality for weighted $L^p$-spaces \cite[Proposition 11.1.3]{gr01} we see that 
\begin{flalign}
\vert V_{\phi_k} \phi_m (\lambda - \mu) \vert &= \frac{1}{\nu_s(\lambda - \mu)}\vert V_{\phi_k} \phi_m (\lambda - \mu) \vert \nu_s(\lambda - \mu) \notag \\
&\leq \frac{1}{\nu_s(\lambda - \mu)}\Vert V_{\phi_k} \phi_m \Vert_{L^{\infty}_{\nu_s}} \notag \\
&\leq \frac{1}{\nu_s(\lambda - \mu)} \Vert V_{\varphi_0} \phi_m \Vert_{L^{\infty}_{\nu_s}} \Vert V_{\phi_k} \varphi_0 \Vert_{L^1_{\nu_s}} \notag \\
&=\frac{1}{\nu_s(\lambda - \mu)}\Vert V_{\varphi_0} \phi_m \Vert_{L^{\infty}_{\nu_s}} \Vert V_{\varphi_0} \phi_k \Vert_{L^1_{\nu_s}} \notag \\
&\leq \frac{C_2}{\nu_s(\lambda - \mu)}\Vert \phi_m \Vert_{M^{1}_{\nu_s}} \Vert \phi_k \Vert_{M^1_{\nu_s}}. \notag
\end{flalign}

\noindent \textbf{Step 4.} We show that $\mathcal{W}$ is intrinsically $s$-localized. We start estimating
\begin{flalign}
&\Vert W_{\lambda} W_{\mu}^*\Vert_{\B(\HS(L^2))}^2 \notag \\
&= \Vert L_{\Psi^*\pi(\lambda)^*\pi(\mu)\Psi}\Vert_{\B(\HS(L^2))}^2 \notag \\
&\leq \Vert \Psi^* \pi(\lambda)^* \pi(\mu) \Psi\Vert_{\HS(L^2)}^2 \notag \\    
&= \sum_{n=1}^{\infty} \left\Vert \left( \sum_{k=1}^{\infty} \psi_k \otimes \phi_k \right) \pi(\lambda)^* \sum_{m=1}^{\infty} \langle e_n, \psi_m \rangle \pi(\mu) \phi_m \right\Vert_{L^2}^2 \notag \\
&= \sum_{n=1}^{\infty} \left\Vert  \sum_{m=1}^{\infty} \sum_{k=1}^{\infty} \langle e_n, \psi_m \rangle \langle \pi(\mu) \phi_m, \pi(\lambda) \phi_k \rangle \psi_k \right\Vert_{L^2}^2 \notag \\
&= \sum_{n=1}^{\infty} \left\langle \sum_{m,k=1}^{\infty} \langle \pi(\mu)\phi_m, \pi(\lambda)\phi_k \rangle \langle e_n, \psi_{m} \rangle \psi_{k}, \sum_{m', k'=1}^{\infty} \langle \pi(\mu)\phi_{m'}, \pi(\lambda)\phi_{k'}\rangle \langle e_n, \psi_{m'} \rangle \psi_{k'} \right\rangle \notag \\
&= \sum_{n=1}^{\infty} \sum_{m,m',k,k' =1}^{\infty} \langle \pi(\mu)\phi_m, \pi(\lambda)\phi_k \rangle \langle e_n, \psi_{m} \rangle \langle  \pi(\lambda)\phi_{k'}, \pi(\mu)\phi_{m'} \rangle \langle \psi_{m'}, e_n \rangle \langle \psi_{k}, \psi_{k'} \rangle \notag \\
&= \sum_{m,m',k,k' =1}^{\infty} \langle \pi(\mu)\phi_m, \pi(\lambda)\phi_k \rangle \langle  \pi(\lambda)\phi_{k'}, \pi(\mu)\phi_{m'} \rangle \langle \psi_{k}, \psi_{k'} \rangle 
\sum_{n=1}^{\infty}  \langle e_n, \psi_{m} \rangle  \langle \psi_{m'}, e_n \rangle  \notag \\
&= \sum_{m,m',k,k' =1}^{\infty} \langle \pi(\mu)\phi_m, \pi(\lambda)\phi_k \rangle  \langle  \pi(\lambda)\phi_{k'}, \pi(\mu)\phi_{m'} \rangle \langle \psi_{k}, \psi_{k'} \rangle \langle \psi_{m'}, \psi_{m} \rangle \notag \\
&\leq \sum_{m,m',k,k' =1}^{\infty} \vert \langle \pi(\mu)\phi_m, \pi(\lambda)\phi_k \rangle \vert \vert \langle  \pi(\lambda)\phi_{k'}, \pi(\mu)\phi_{m'} \rangle \vert \notag \\
&= \left( \sum_{m=1}^{\infty} \sum_{k=1}^{\infty}\vert V_{\phi_k} \pi(\mu)\phi_m (\lambda) \vert \right)^2 ,\notag
\end{flalign}
where changing the order of summation is justified as in Step 1 via absolute convergence of the sum in question. By the covariance property of the STFT \cite{grochenigtfa} and Step 3, we may conclude from the latter that
\begin{flalign}
\Vert W_{\lambda} W_{\mu}^*\Vert_{\B(\HS(L^2))} &\leq \sum_{m=1}^{\infty} \sum_{k=1}^{\infty}\vert V_{\phi_k} \pi(\mu)\phi_m (\lambda) \vert \notag \\
&= \sum_{m=1}^{\infty} \sum_{k=1}^{\infty}\vert V_{\phi_k} \phi_m (\lambda - \mu) \vert \notag \\
&\leq \frac{C_2}{\nu_s(\lambda - \mu)} \sum_{m=1}^{\infty} \sum_{k=1}^{\infty} \Vert \phi_m \Vert_{M^{1}_{\nu_s}} \Vert \phi_k \Vert_{M^1_{\nu_s}} \notag \\
&= \frac{C_2 C_3^2}{\nu_{s}(\lambda - \mu)}, \notag 
\end{flalign}
as was to be shown. 

\noindent \textbf{Step 5.} We show that $\mathcal{G}$ and $\mathcal{W}$ are mutually $s$-localized. As in Step 4, we have \begin{flalign}
\Vert G_{\lambda} W_{\mu}^*\Vert_{\B(\HS(L^2))}^2 &\leq \sum_{n=1}^{\infty} \left\Vert  \sum_{m=1}^{\infty} \langle e_n, \psi_m \rangle \langle \pi(\mu) \phi_m, \pi(\lambda) \varphi_0 \rangle \varphi_0 \right\Vert_{L^2}^2 \notag \\
&= \sum_{n=1}^{\infty} \sum_{m,m'=1}^{\infty} \langle \pi(\mu)\phi_m, \pi(\lambda)\varphi_0 \rangle \langle e_n, \psi_{m} \rangle \langle  \pi(\lambda)\varphi_0, \pi(\mu)\phi_{m'} \rangle \langle \psi_{m'}, e_n \rangle \notag \\
&= \sum_{m,m' =1}^{\infty} \langle \pi(\mu)\phi_m, \pi(\lambda)\varphi_0 \rangle  \langle  \pi(\lambda)\varphi_0, \pi(\mu)\phi_{m'} \rangle  \langle \psi_{m'}, \psi_{m} \rangle \notag \\
&\leq \left( \sum_{m=1}^{\infty} \vert V_{\varphi_0} \pi(\mu)\phi_m (\lambda) \vert \right)^2 \notag
\end{flalign}
and therefore
$$\Vert G_{\lambda} W_{\mu}^*\Vert_{\B(\HS(L^2))} \leq \sum_{m=1}^{\infty} \vert V_{\varphi_0} \phi_m (\lambda - \mu) \vert \leq \frac{C_2 C_3}{\nu_{s}(\lambda - \mu)},$$
as was to be shown. 

\noindent \textbf{Step 6.} By Step 4 and the preliminary results from Subsection \ref{Polynomially localized g-frames}, the co-orbit spaces $\Co_m^p(\mathcal{W})$ ($\frac{2d}{s-r} <p \leq \infty$) are well-defined and non-trivial. Since, by Step 5, $\mathcal{G}$ and $\mathcal{W}$ are mutually $s$-localized, we may conclude from Theorem \ref{samecoorbitspaces} that $\Co_m^p(\mathcal{W}) = \Co_m^p(\mathcal{G})$ with equivalent norms ($\frac{2d}{s-r}<p \leq \infty$). The last part of the statement now follows from Theorem \ref{mainchar} and Definition \ref{Mpdef}.
\end{proof}

\begin{remark}\label{specialpsi} 
1.) An operator $\Psi$ as in the above theorem is a so-called \emph{nuclear operator} $\Psi \in \mathcal{N}(L^2(\mathbb{R}^d);M^1_{\nu_s}(\mathbb{R}^d))$ as considered e.g. in \cite{doelumcskr24}, where it was shown that $\mathcal{N}(L^2(\mathbb{R}^d);M^1_{\nu_s}(\mathbb{R}^d)) \subset \mathfrak{M}^1_{\nu_s}(\mathbb{R}^d) \subset \HS(L^2(\mathbb{R}^d))$. Furthermore, in \cite[Theorem 6.8]{doelumcskr24}, sufficiency conditions for $(\Psi^*\pi(\lambda)^*)_{\lambda \in \Lambda}$ being a g-frame for $L^2(\mathbb{R}^d)$ are given in terms of its \emph{oscillation kernel}. 

2.) If $\Psi$ in the above theorem has finite rank, that is, $\Psi= \sum_{m=1}^{K} \phi_m \otimes \psi_m$ with $\phi_m \in M_{\nu_s}^{1}(\mathbb{R}^d)$ and $\psi_m \in L^2(\mathbb{R}^d)$ ($m=1, \dots, K$), then the assumption $(\Psi^* \pi(\lambda)^*)_{\lambda \in \Lambda}$ being a g-frame for $L^2(\mathbb{R}^d)$ is equivalent to $\lbrace \phi_1, \dots , \phi_K \rbrace$ generating a multi-window Gabor frame for $L^2(\mathbb{R}^d)$ over the lattice $\Lambda$. In particular, if the functions $\phi_1 = \dots = \phi_K =: \phi$ coincide, this amounts to $(\pi(\lambda)\phi)_{\lambda \in \Lambda}$ being a Gabor frame for $L^2(\mathbb{R}^d)$. 

3.) In view of the approximation results in the upcoming Section \ref{Approximation theoretic results}, we prefer the localization parameter $s$ to be large, since larger $s$ amounts to smaller possible $p$ in the definition of $\Co_m^p(\mathcal{W})$. Therefore, $\phi_m \in S(\mathbb{R}^d)$ ($m\in \mathbb{N}$) being Schwartz functions is the optimal choice in this regard, since they imply polynomial decay of arbitrary order. In this case we obtain Theorem \ref{mainchar2} (as well as Corollary \ref{Mpeasy} with $\Psi$ instead of $\Phi_0$) for the full range of $0<p\leq 2$. 
\end{remark}

\begin{remark}\label{differentanalysisoperators}
We explicitly clarify the relation between the several different analysis operators encountered so far. More precisely, let $\Phi \in \mathfrak{M}^1_{\nu_s}(\mathbb{R}^d)$ and assume that $\mathcal{U} = (U_{\lambda})_{\lambda \in \Lambda} = (L_{\Phi^*\pi(\lambda)^*})_{\lambda\in \Lambda}$ is a g-frame for $\HS(L^2(\mathbb{R}^d))$. Then, by Theorem \ref{mainchar2}, the analysis operator $C_{\Phi}$ from Proposition \ref{analbound} coincides with the (g-frame) analysis operator $C_{\mathcal{U}}$ introduced in Subsection \ref{Some background on g-frames and localization} and their respective action on $F\in \mathfrak{M}_m^p(\mathbb{R}^d)$ equals the operator STFT $\mathcal{V}_{\Phi} F$ sampled on $\Lambda$, that is, 
$$C_{\Phi} F = C_{\mathcal{U}} F = (U_{\lambda}F)_{\lambda \in \Lambda} =  (\Phi^*\pi(\lambda)^*F)_{\lambda \in \Lambda} = (\mathcal{V}_{\Phi} F(\lambda))_{\lambda\in \Lambda}.$$
\end{remark}

We are going to look at two examples of operator classes, whose operator short-time Fourier transform (STFT) admits an intuitive interpretation. As we will see, these operators also satisfy structural conditions that lead to accurate approximation using only a small number of operator building blocks. In other words, their most significant operator STFT coefficients yield sparse and effective representations.

\section{Examples}\label{Sec:Exa1} 
We give two examples to illustrate the nature of the introduced transforms. 
\begin{example}\label{Ex1}
 Let \( H_\eta \) be a linear operator on \( L^2(\mathbb{R}) \) defined via its Fourier-Wigner transform   \( \eta  = \mathcal{F}_W (H)\in L^2(\mathbb{R}^2) \), i.e. we formally write
\[
H_\eta = \int_{\mathbb{R}^2} \eta(z) \pi(z) \, dz,
\]
where \( z = (x,\xi) \in \mathbb{R}^2 \), and if we assume \( \eta \) is compactly supported in a set \( \Omega_\eta \subset \mathbb{R}^2 \), the integral is well-defined (e.g., in the weak operator topology).

The interesting observation about this operator is the fact, that it has infinite rank. In this sense, it is particularly interesting to consider its approximation by different operator expansion. 

We fix  a window   \( \varphi_0 \in L^2(\mathbb{R}) \)  (typically a Gaussian) and analyze \( H_\eta \) with respect to the rank-one analysis operator window \( \Phi_0 = \varphi_0\otimes\varphi_0\):
\[
\mathcal{V}_{\Phi_0}H_\eta(\lambda) = C_{\Phi_0}(H_\eta)(\lambda) = \varphi_0 \otimes H_\eta^* \pi(\lambda) \cdot \varphi_0
\]
Applying this operator to a test function \( \psi \in L^2(\mathbb{R}) \), we obtain:
\[[C_{\Phi_0}(H_\eta)(\lambda)]\,\psi
= \left\langle  H_\eta \psi,\pi(\lambda) \varphi_0 \right\rangle \varphi_0.\]
Setting \( V_{\varphi_0} \psi(z) = \langle \psi, \pi(z) \varphi_0 \rangle \), then under mild assumptions (e.g., Schwartz functions), we can write the inner product as:
\[\left\langle  H_\eta \psi,\pi(\lambda) \varphi_0 \right\rangle = \int_{\mathbb{R}^2} \eta(z) \langle \psi ,\pi(z)^* \pi(\lambda) \varphi_0\rangle \, dz 
= \int_{\mathbb{R}^2} \eta(z) V_{\varphi_0} \psi(\lambda - z) e^{-2\pi i (x (l-\xi))} \, dz,\]
Writing this as a twisted convolution
we obtain 
\(C_{\Phi_0}(H_\eta)(\lambda)\psi=[\eta \,\natural\, V_{\varphi_0} \psi](\lambda) \cdot \varphi_0\)
and each coefficient operator \( C_{\Phi_0}(H_\eta)(\lambda) \) is a rank-one operator with  Hilbert–Schmidt norm 
\[
\|C_{\Phi_0}(H_\eta)(\lambda)\|_{\mathcal{HS}} 
= \sup_{\|\psi\|_{L^2}=1} \left| [\eta \,\natural\, V_{\varphi_0} \psi](\lambda) \right|.
\]

Thus, the operator STFT is pointwise bounded by the twisted convolution of the spreading function \( \eta \) with the STFT of \( \psi \), maximized over unit-norm signals.
If \( \eta \) is smooth and compactly supported, the twisted convolution with \( V_{\varphi_0} \psi \) is smooth and localized. 
In Figure~\ref{Fig:CompSpread1}, we show two instances of 
$\mathcal{V}_{\Phi}F(\lambda)$ for an operator $F = H_\eta$, which is also depicted in its matrix form. The last plot shows the norms \(\|C_{\Phi_0}(H_\eta)(\lambda)\|_{\mathcal{HS}} \) for all $\lambda\in\Lambda$.
\begin{figure}
    \centering
\includegraphics[width=0.9\linewidth]{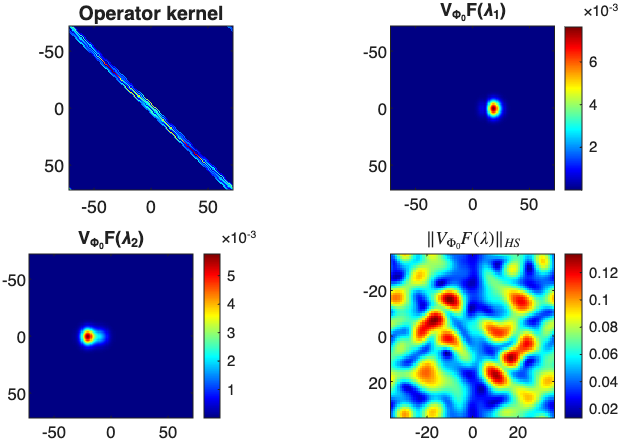}
    \caption{Operator STFT of an underspread operator $H_\eta$: Upper left plot shows the operator kernel of an operator $F = H_\eta$, upper right and lower left show two instances of the operators
$\mathcal{V}_{\Phi}F(\lambda_i)$ and the lower right  plot shows the norms \(\|C_{\Phi_0}(H_\eta)(\lambda)\|_{\mathcal{HS}} \) on a sub-sampled version of  $\Lambda$.}
    \label{Fig:CompSpread1}
\end{figure}


In a next step, we compute this explicitly for a Gaussian spreading function \( \eta \):
\(\eta(x,\xi) = e^{-\pi (x^2 + \xi^2)}\).
The resulting operator $H_\eta =  \int_{\mathbb{R}^2} e^{-\pi |z|^2} \pi(z)\, dz$ is  trace-class operator with rapidly decaying spectrum.
Using a rank-one analysis operator with a Gaussian window \( \varphi_0 \in \mathcal{S}(\mathbb{R}) \), normalized in \( L^2 \), we obtain:
\[
C_{\Phi_0}(H_\eta)(\lambda)\psi(x) = \left[\eta \,\natural\, V_{\varphi_0} \psi\right](\lambda) \cdot \varphi_0(x),
\]
 with Hilbert–Schmidt norm of this rank-one operator is:
\[
\|C_{\Phi_0}(H_\eta)(\lambda)\|_{\mathcal{HS}} 
= \sup_{\|\psi\|=1} \left| \eta \,\natural\, V_{\varphi_0} \psi(\lambda) \right|.
\]

Since both \( \eta \) and \( V_{\varphi_0} \psi \) are Gaussians, the twisted convolution \( \eta \,\natural\, V_{\varphi_0} \psi \) is also a Gaussian, centered near the origin.

Therefore, \( \|C_{\Phi_0}(H_\eta)(\lambda)\|_{\mathcal{HS}} \) decays exponentially in \( |\lambda| \),  and we expect  sparsity of the operator representation.
\end{example}

\begin{remark}
The supremum
\[
\|C_{\Phi_0}(H_\eta)(\lambda)\|_{\mathcal{HS}} = \sup_{\|\psi\|=1} \left| \eta \natural V_{\varphi_0} \psi(\lambda) \right|
\]
is in general not attained by a specific choice of $\psi$. However, when both $\eta$ and the window $\varphi_0$ are Gaussians, the short-time Fourier transform $V_{\varphi_0} \varphi_0$ is also a Gaussian, and in this case the supremum is attained at $\psi = \varphi_0$, up to a phase. Therefore, for Gaussian operators and windows, we have
\[
\|C_{\Phi_0}(H_\eta)(\lambda)\|_{\mathcal{HS}} = |\eta \natural V_{\varphi_0} \varphi_0 (\lambda)|.
\]    
\end{remark}

\begin{example}\label{Ex2}
The next class of operators we consider consists of time-frequency localization operators, whose spreading function is given by  $\eta_m = \mathcal{M} \cdot V_g g $ with $\mathcal{M}  =  \mathcal{F}_s(m)$, where $\mathcal{F}_s$ denotes the symplectic Fourier transform. For a window $g$, the localization operator can be written as convolution, cf.~\cite{skrett18}, so that 
\(
H_{m} = m \star \Gamma \quad (\Gamma = g \otimes g)
\)
and since
\(
 C_{\Phi_0}(H_{m})(\lambda) = \varphi_0 \otimes H_{m} \pi(\lambda) \varphi_0,
\)
we have 
\[
 \|C_{\Phi_0}(H_{m})(\lambda)\|_{{\mathcal{HS}}} = \|H_{m} \pi(\lambda ) \varphi_0\|_{L^2} \leq \left\| m \cdot T_\lambda V_g \varphi_0 \right\|_{L^2}
\]
This means, if $g$ is TF-localized,  here, then the relevant coefficients should be given by $\lambda\in\Lambda$, for which  $|m(\lambda )|$ has higher values. 
A numerical example verifies the observation: 
in Figure~\ref{Fig:CompGabMul1} we show the operator kernel, the mask $m$, which localizes energy in two quadratic domains of phase space and  the norms \(\|C_{\Phi_0}(H_m)(\lambda)\|_{\mathcal{HS}} \) for all $\lambda\in\Lambda$.
\begin{figure}
    \centering
\includegraphics[width=0.9\linewidth]{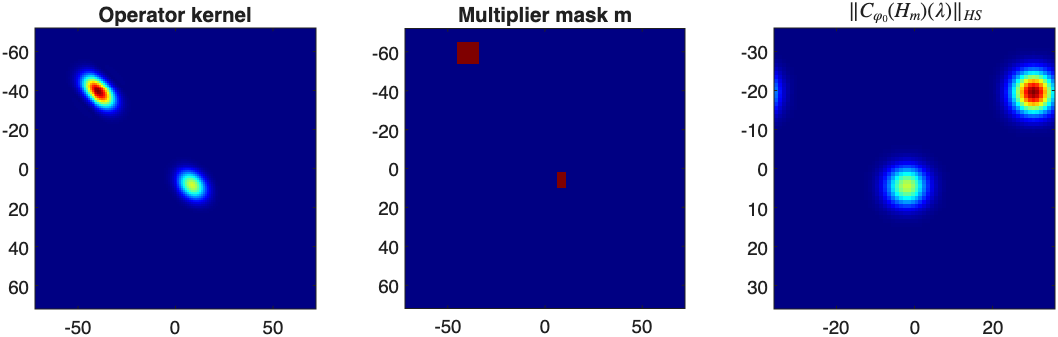}
    \caption{Operator STFT of Localization operator $H_m$.}
    \label{Fig:CompGabMul1}
\end{figure}

\end{example}
\section{Approximation theoretic results}\label{Approximation theoretic results}

Here we define dictionaries of bounded operators on $\HS(L^2(\mathbb{R}^d))$ and their associated approximation spaces and relate them to the operator coorbit spaces $\mathfrak{M}^p$. 

\subsection{Operator dictionaries and approximation spaces}

For simplicity and generality, we first introduce operator dictionaries and their associated approximation spaces for an abstract Hilbert space $\Hil$.

\begin{definition}\label{defdictionary}
A countable family $\mathcal{U} = (U_{\lambda})_{\lambda \in \Lambda}$ in $\B(\Hil)$ is called an \emph{operator dictionary} for $\Hil$, if 
\begin{equation}\label{dictionary}
    D_{\mathcal{U}} (c_{00}(\Lambda;\Hil)) = \left\lbrace \sum_{\lambda \in \Lambda} U_{\lambda}^* g_{\lambda} : (g_{\lambda})_{\lambda \in \Lambda}\in c_{00}(\Lambda;\Hil) \right\rbrace
\end{equation}
is dense in $\Hil$.
\end{definition}

\begin{example}\label{gframedic}
Every g-frame $\mathcal{U} = (U_{\lambda})_{\lambda \in \Lambda}$ for $\Hil$ is an operator dictionary for $\Hil$. Indeed, this follows immediately from $D_{\mathcal{U}}:\ell^2(\Lambda;\Hil)\longrightarrow \Hil$ being bounded and surjective and $c_{00}(\Lambda;\Hil)$ being dense in $\ell^2(\Lambda;\Hil)$.    
\end{example}

Closely related to an operator dictionary $\mathcal{U}$ are the sets $\Sigma_m(\mathcal{U})$ ($m\in \mathbb{N}_0$) defined below.

\begin{definition}\label{defsigmam}
Let $\mathcal{U} = (U_{\lambda})_{\lambda \in \Lambda}$ be an operator dictionary for $\Hil$. Then for every $m\in \mathbb{N}$, the set $\Sigma_m(\mathcal{U})$ is defined as
\begin{equation}
\Sigma_m(\mathcal{U}) := \left\lbrace \sum_{\lambda \in K} U_{\lambda}^* g_{\lambda} : g_\lambda \in \Hil \, (\forall \lambda \in \Lambda), \, K \subset \Lambda, \, \vert K\vert \leq m \right\rbrace .
\end{equation}
We also define $\Sigma_0(\mathcal{U}):= \lbrace 0_{\Hil} \rbrace$.
\end{definition}

\begin{definition}\label{mtermapproxerror}
Given an operator dictionary $\mathcal{U} = (U_{\lambda})_{\lambda \in \Lambda}$ for $\Hil$, $f\in \Hil$ 
and $m\in \mathbb{N}_0$, the \emph{best $m$-term approximation error} of $f$ is defined by 
$$\sigma_m(f,\mathcal{U}) := \inf_{h\in \Sigma_m(\mathcal{U})} \Vert f-h\Vert_{\Hil}.$$
Furthermore, for $0<\alpha < \infty$ and $0<p\leq \infty$, we define the \emph{approximation space} $\A_p^{\alpha}(\mathcal{U})$ by
$$\A_p^{\alpha}(\mathcal{U}) := \left\lbrace f\in \Hil: \Vert f \Vert_{\A_p^{\alpha}(\mathcal{U})} := \left( \sum_{m=1}^{\infty} \big( m^{\alpha} \sigma_{m-1}(f,\mathcal{U}) \big)^p \frac{1}{m} \right)^{\frac{1}{p}} < \infty \right\rbrace$$
and analogously $\A_{\infty}^{\alpha}(\mathcal{U})$ via the natural modification 
$$\Vert f \Vert_{\A_{\infty}^{\alpha}(\mathcal{U})} := \sup_{m\in \mathbb{N}} m^{\alpha} \sigma_{m-1}(f,\mathcal{U}) $$
in the case $p=\infty$.
\end{definition}

The following basic properties are easy to verify, hence we omit the proof.

\begin{lemma}\label{basicproperties}
Let $\mathcal{U} = (U_{\lambda})_{\lambda \in \Lambda}$ be an operator dictionary for $\Hil$. Then for $f,g\in \Hil$, $\lambda \in \mathbb{C}$ and $m\in \mathbb{N}_0$ the following hold:
\begin{itemize}
    \item[(i)] $\Sigma_m(\mathcal{U}) \subset \Sigma_{m+1}(\mathcal{U})$;
    \item[(ii)] $\lambda \Sigma_m(\mathcal{U}) = \Sigma_m(\mathcal{U})$ in case $\lambda \neq 0$;
    \item[(iii)] $\Sigma_{m}(\mathcal{U})+\Sigma_{m}(\mathcal{U}) \subset \Sigma_{2m}(\mathcal{U})$;
    \item[(iv)] $\sigma_{m+1}(f,\mathcal{U}) \subset \sigma_{m}(f,\mathcal{U})$;
    \item[(v)] $\sigma_{m}(\lambda f,\mathcal{U}) = \vert \lambda \vert \sigma_{m}(f,\mathcal{U})$;
    \item[(vi)] $\sigma_{2m}(f+g,\mathcal{U}) \leq \sigma_{m}(f,\mathcal{U}) + \sigma_{m}(g,\mathcal{U})$.
\end{itemize}
\end{lemma}

Having established Lemma \ref{basicproperties}, the following result follows easily from a general principle \cite[p.234]{DeVore1993ConstructiveA}. For more details we also refer to the master's thesis \cite[Lemma 3.5]{hol22}, where the special case of rank one operators $U_{\lambda}$ is treated.

\begin{proposition}\label{contembedding}
For each $0<\alpha <\infty$ and each $0<p\leq \infty$, $\A_p^{\alpha}(\mathcal{U})$ is a linear subspace of $\Hil$ and $\Vert \cdot \Vert_{\A_p^{\alpha}(\mathcal{U})}$ defines a quasi-norm on $\A_p^{\alpha}(\mathcal{U})$.
\end{proposition}

The next result is shown as in \cite[p.234]{DeVore1993ConstructiveA} as well.

\begin{proposition}
Suppose that either $\alpha > \beta >0$ or $\alpha = \beta >0$ and $0< p\leq q \leq \infty$. Then $\A_p^{\alpha}(\mathcal{U})$ is continuously embedded in $\A_{q}^{\beta}(\mathcal{U})$, i.e. there exists $C>0$ such that 
$$\Vert f\Vert_{\A_{q}^{\beta}(\mathcal{U})} \leq C \Vert f\Vert_{\A_p^{\alpha}(\mathcal{U})} \qquad (\forall f\in \Hil).$$
\end{proposition}

\subsection{Best $m$-term approximation}

Next, we apply the results from the previous subsection to the Hilbert space $\HS = \HS(L^2(\mathbb{R}^d))$ with operator dictionary $\mathcal{W}$ given by a g-frame $(W_{\lambda})_{\lambda \in \Lambda}$ as in Theorem \ref{mainchar2}. We call such a dictionary an \emph{admissible dictionary}.

\begin{definition}[Admissible dictionaries]
Let $s>2d$ and $\Psi$ be an operator of the form 
$$\Psi= \sum_{m=1}^{\infty} \phi_m \otimes \psi_m$$
with $\phi_m \in 
M^1_{\nu_s}(\mathbb{R}^d)$, $\psi_m \in L^2(\mathbb{R}^d)$ for all $m\in \mathbb{N}$ and
$$\sum_{m=1}^{\infty} \Vert \phi_m \Vert_{M_{\nu_s}^{1}} \Vert \psi_m \Vert_{L^2} <\infty .$$
Assume that $(\Psi^* \pi(\lambda)^*)_{\lambda \in \Lambda}$ is a g-frame for $L^2(\mathbb{R}^d)$, where $\Lambda$ is the a lattice as in Theorem \ref{polydecayprop}. Then we call 
$$\mathcal{W} = (W_{\lambda})_{\lambda \in \Lambda}, \qquad W_{\lambda} := L_{\Psi^* \pi(\lambda)^*} \quad (\lambda \in \Lambda)$$
an \emph{$s$-admissible dictionary}. In case $\mathcal{W}$ is an $s$-admissible dictionary for all $s>2d$, we call $\mathcal{W}$ an \emph{admissible dictionary}.
\end{definition}

\begin{example}[Examples of admissible dictionaries]
Our toy example of an admissible dictionary is the g-frame $\mathcal{G} = (L_{\Phi_0 \pi(\lambda)^*})_{\lambda \in \Lambda}$ from Theorem \ref{polydecayprop} associated with the Gaussian state $\Phi_0^* = \Phi_0 = \varphi_0 \otimes \varphi_0$. As noted in Remark \ref{specialpsi}, any finite-rank operator $\Psi= \sum_{m=1}^{K} \phi_m \otimes \psi_m$ with $\phi_m \in M^1_{\nu_s}(\mathbb{R}^d)$ and $\psi_m \in L^2(\mathbb{R}^d)$ ($m=1, \dots, K$) such that $\lbrace \phi_1, \dots , \phi_K \rbrace$ generates a multi-window Gabor frame for $L^2(\mathbb{R}^d)$ yields an $s$-admissible dictionary. In particular, if $(\pi(\lambda)\phi)_{\lambda \in \Lambda}$ ($\phi \in S(\mathbb{R}^d)$) is a Gabor frame for $L^2(\mathbb{R}^d)$, $\Psi= \sum_{m=1}^{K} \phi \otimes \psi_m$ yields an $s$-admissible dictionary. If $\phi_1, \dots , \phi_K \in S(\mathbb{R}^d)$, then the latter two are examples of admissible dictionaries.
\end{example}

We follow ideas from \cite{CORDERO200429}, see also \cite{hol22}, to derive results on approximation properties of the spaces $\mathfrak{M}^p(\mathbb{R}^d)$. The following lemma from Stechkin \cite{Ste51} and DeVore and Temlyakov \cite{DeVore1996} is crucial for our analysis.

\begin{lemma}\label{approxtool}
Let $a = \lbrace a_n\rbrace_{n=1}^{\infty}$ be a sequence such that $a_1 \geq a_2 \geq a_3 \geq \dots \geq 0$, let $0<p<2$, and set 
$$\alpha = \frac{1}{p}-\frac{1}{2} \qquad \text{and} \qquad \sigma_m(a) = \left( \sum_{j=m+1}^{\infty} \vert a_j\vert^2\right)^{\frac{1}{2}}.$$
Then there exists a constant $C=C(p)>0$, such that
$$\frac{1}{C} \Vert a \Vert_{\ell^p(\mathbb{N})} \leq \left( \sum_{m=1}^{\infty} \big( m^{\alpha} \sigma_{m-1}(a) \big)^p \frac{1}{m} \right)^{\frac{1}{p}} \leq C  \Vert a \Vert_{\ell^p(\mathbb{N})}.$$
\end{lemma}

\begin{theorem}\label{approxmain}
Let $\mathcal{W}$ be an $s$-admissible dictionary (respectively admissible dictionary) and $\alpha = \frac{1}{p}-\frac{1}{2}$, where $\frac{2d}{s}<p<2$ (respectively $0<p<2$). Then $\mathfrak{M}^p(\mathbb{R}^d)$ is continuously embedded in $\A_p^{\alpha}(\mathcal{W})$. 
\end{theorem}

\begin{proof}
By Corollary \ref{Mpeasy}, $\mathfrak{M}^p(\mathbb{R}^d) = \left\lbrace F\in \HS(L^2(\mathbb{R}^d)): (W_{\lambda} F )_{\lambda \in \Lambda} \in \ell^p(\Lambda;\HS)  \right\rbrace$ and $\Vert F \Vert_{\mathfrak{M}^p} \asymp \Vert (W_{\lambda} F)_{\lambda \in \Lambda} \Vert_{\ell^p(\Lambda;\HS)}$. So, let $F\in \mathfrak{M}^p(\mathbb{R}^d)$ with $(W_{\lambda} F)_{\lambda \in \Lambda} \in \ell^p(\Lambda;\HS)$, and let $\tau:\mathbb{N}\longrightarrow \Lambda$ be a bijection so that 
$$\Vert W_{\tau(1)} F \Vert_{\HS} \geq \Vert W_{\tau(2)} F \Vert_{\HS} \geq \Vert W_{\tau(3)} F \Vert_{\HS} \geq \dots \geq 0.$$
For $m\in \mathbb{N}$, set $P_m F :=  \sum_{j=1}^m (\widetilde{W}_{\tau(\lambda)})^* W_{\tau(\lambda)} F$. Since, by g-frame reconstruction, $F= \sum_{\lambda \in \Lambda} (\widetilde{W}_{\lambda})^* W_{\lambda} F$ with unconditional convergence in $\HS$, we may estimate the best $m$-term approximation error of $F$ with respect to $\mathcal{W}$ via 
\begin{flalign}
\sigma_{m}(F,\mathcal{W}) &\leq \Vert F - P_{m}F \Vert_{\HS} \notag \\
&= \left\Vert \sum_{k=m+1}^{\infty} (\widetilde{W}_{\tau(\lambda)})^* W_{\tau(\lambda)} F \right\Vert_{\HS} \notag \\
&\leq \Vert D_{\widetilde{\mathcal{W}}} \Vert_{\B(\ell^2(\Lambda;\HS),\HS)} \left( \sum_{k=m+1}^{\infty} \Vert W_{\tau(\lambda)} F \Vert_{\HS}^2 \right)^{\frac{1}{2}}. \notag 
\end{flalign}
Setting $a_n := \Vert W_{\tau(n)} F \Vert_{\HS}$ and $\sigma_m(a) = \left( \sum_{j=m+1}^{\infty} \vert a_j\vert^2\right)^{\frac{1}{2}}$, we obtain via an application of Lemma \ref{approxtool} that
\begin{flalign}
\Vert F \Vert_{\A_p^{\alpha}(\mathcal{W})} &= \left( \sum_{m=1}^{\infty} \big( m^{\alpha} \sigma_{m-1}(F,\mathcal{W}) \big)^p \frac{1}{m} \right)^{\frac{1}{p}} \notag \\
&\leq \Vert D_{\widetilde{\mathcal{W}}} \Vert_{\B(\ell^2(\Lambda;\HS),\HS)} \left( \sum_{m=1}^{\infty} \big( m^{\alpha} \sigma_{m-1}(a) \big)^p \frac{1}{m} \right)^{\frac{1}{p}} \notag \\
&\leq C' \Vert a \Vert_{\ell^p(\mathbb{N})} \notag \\
&= C' \Vert C_{\mathcal{W}}F \Vert_{\ell^p(\Lambda;\HS)} \notag \\
&\leq C'' \Vert F \Vert_{\mathfrak{M}^p}, \notag
\end{flalign}
where we used the equivalence of the norms $\Vert \cdot \Vert_{\Co^p(\widetilde{\mathcal{W}})} \asymp  \Vert \cdot \Vert_{\Co^p(\mathcal{W})} \asymp \Vert \cdot \Vert_{\mathfrak{M}^p}$ according to Theorems \ref{isometrycor} and \ref{mainchar} in the last step. 
\end{proof}

\begin{remark}
Combining the previous result with Proposition \ref{contembedding} yields 
$$\sup_{m\in \mathbb{N}} m^{\alpha} \sigma_{m-1}(F,\mathcal{W}) = \Vert F\Vert_{\A^{\alpha}_{\infty}(\mathcal{W})} \leq C \Vert F \Vert_{\mathfrak{M}^p}$$
and, therefore, 
$$\sigma_{m}(F,\mathcal{W}) \lesssim \frac{1}{(m+1)^{\alpha}}\Vert F \Vert_{\mathfrak{M}^p} \qquad (m = 0, 1, 2, \dots ).$$
In other words, any operator $F$ from $\mathfrak{M}^p(\mathbb{R}^d)$ (with $\frac{2d}{s}<p<2$) can be approximated by $m$ terms with approximation rate $\alpha = \frac{1}{p}-\frac{1}{2}$. 

In particular, the proof of Theorem \ref{approxmain}
reveals a recipe for achieving this, namely by picking the $m$ largest values of the measurements
$$(\Vert W_{\lambda} F \Vert_{\HS})_{\lambda \in \Lambda} = (\Vert  (\mathcal{V}_{\Psi} F)(\lambda) \Vert_{\HS})_{\lambda \in \Lambda}$$
and compute the associated truncated g-frame reconstruction series $P_m F$.
\end{remark}

\subsection{Sparseness classes}

For mixed-state operators satisfying a mild additional condition, also being an element of the Feichtinger operators, the results in this Section give a sparse description in terms of a representation based on an admissible operator g-frame for the space of Hilbert-Schmidt operators.  
The key to sparse representations of density operators is the extension of \cite[Definition 2]{grib-nielsen2} to the g-frame setting.

\begin{definition}\label{sparsenessclass}
Let $\mathcal{U} =(U_{\lambda})_{\lambda \in \Lambda}$ be a g-frame for a Hilbert space $\Hil$ and suppose that $0<p<\infty$ and a given weight $m$ are given so that $\ell^p_m(\Lambda;\Hil)$ is continuously embedded in $\ell^2(\Lambda;\Hil)$. Then the \emph{sparseness class} $S^p_m(\mathcal{U})$ is defined by  
$$S^p_m(\mathcal{U}) = \left\lbrace f\in \Hil: \exists (g_{\lambda})_{\lambda \in \Lambda}\in \ell_m^p(\Lambda;\Hil): f=\sum_{\lambda \in \Lambda} U_{\lambda}^* g_{\lambda}\right\rbrace$$
and the associated norm $\Vert \cdot \Vert_{S^p_m(\mathcal{U})}$ is defined by
$$\Vert f \Vert_{S^p_m(\mathcal{U})} = \inf\left\lbrace \Vert (g_{\lambda})_{\lambda \in \Lambda} \Vert_{\ell^p_m(\Lambda;\Hil)} : f = \sum_{\lambda \in \Lambda} U_{\lambda}^* g_{\lambda}\right\rbrace .$$
\end{definition}

Once again, we employ the theory of localized g-frames (in the case $\Hil = \HS(L^2(\mathbb{R}^d))$) and show that $S^p_m(\mathcal{W}) = \mathfrak{M}_m^p(\mathbb{R}^d)$ for admissible dictionaries $\mathcal{W}$.

\begin{theorem}
Let $\mathcal{W} =(W_{\lambda})_{\lambda \in \Lambda}$ be an $s$-admissible dictionary (respectively admissible dictionary) and let $\frac{2d}{s}<p<\infty$ (respectively $0<p<\infty$) and $m$ be as in Definition \ref{sparsenessclass}. Then 
$$S^p_m(\mathcal{W}) = \mathfrak{M}_m^p(\mathbb{R}^d)$$
with equivalent norms.
\end{theorem}

\begin{proof}
This follows immediately from Theorem \ref{mainchar2} and Theorem \ref{Donto}.
\end{proof}

Operators in the operator Feichtinger algebra $ \mathfrak{M}^{1}(\mathbb{R}^d)$ are trace class operators and those with unit trace are the class of mixed states, for which the preceding theorem guarantees a sparse representation. This seems to be the first result on sparse representations of quantum states, and we will explore its implications for problems in quantum science in future work, such as quantum machine learning and quantum communication. This is going to pave the way to operator version of well-established compressed sensing algorithms for signals, which should be of interest from an applied and pure aspect.

\section{Numerical simulations}

In this section, we first revisit the examples introduced in Section~\ref{Sec:Exa1} and then look at examples obtained from sampling random density matrices, which simulate quantum states. In each case, we discuss if and how the respective density operators fulfill assumptions necessary for a sparse approximation. 

\subsection{Density Operators with compactly supported spreading function }

We revisit Example~\ref{Ex1} and investigate best \( K \)-term approximations of the operator \( H_\eta \) based on its operator STFT coefficients \( C_{\Phi_0}(H_\eta)(\lambda) = \mathcal{V}_{\Phi_0}H_\eta(\lambda)  \). Specifically, we sort the coefficients in decreasing order of their Hilbert--Schmidt norms \(\| C_{\Phi_0}(H_\eta)(\lambda) \|_\HS\) , and retain only the top \( K \) terms.  We denote the index set containing  the indices of the \( K \) largest operator STFT coefficients by  \( \Lambda_m \subset \Lambda \) and obtain   a sparse approximation
\[
H_\eta^{(K)} := \sum_{\lambda \in \Lambda_m} C_{\Phi_0}(H_\eta)(\lambda) \cdot \pi(\lambda)^* \Phi,
\]

We report approximation errors for different values of \( K \in \{20, 30, 100, 500\} \) using the normalized Frobenius norm:
\[
\text{AppErr}_m := \frac{\| H_\eta^{(K)} - H_\eta \|_{\mathrm{F}}}{\| H_\eta \|_{\mathrm{F}}}.
\]
Numerical results demonstrate that even relatively small values of \( K \) yield good approximations due to the sparsity of the spreading function. Figure~\ref{Fig:KtermApp} shows the reconstruction quality and coefficient structure for several values of \( K \).
Observe that the full representation uses $5184$ coefficients.
\begin{figure}[h!]
    \centering
\includegraphics[width=\textwidth]{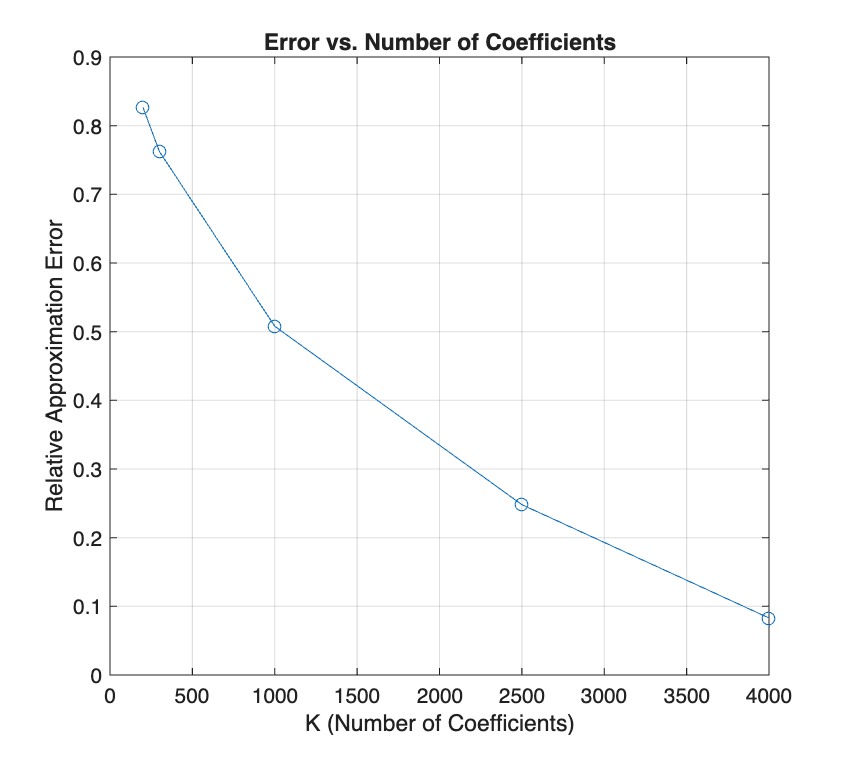}
    \caption{Approximation results using best \(K\)-term expansions of the operator STFT.}
    \label{Fig:KtermApp}
\end{figure}
\begin{figure}[ht]
    \centering
    \subfigure[K=20]{\includegraphics[width=0.3\linewidth]{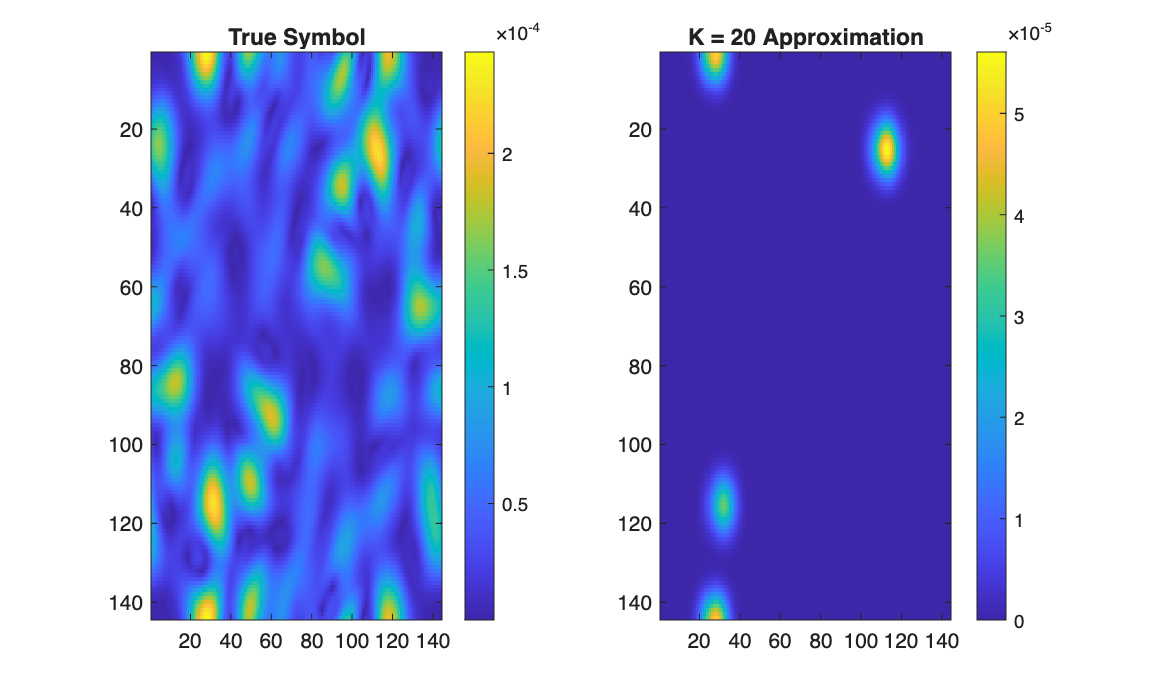}}
    \subfigure[K=100]{\includegraphics[width=0.3\linewidth]{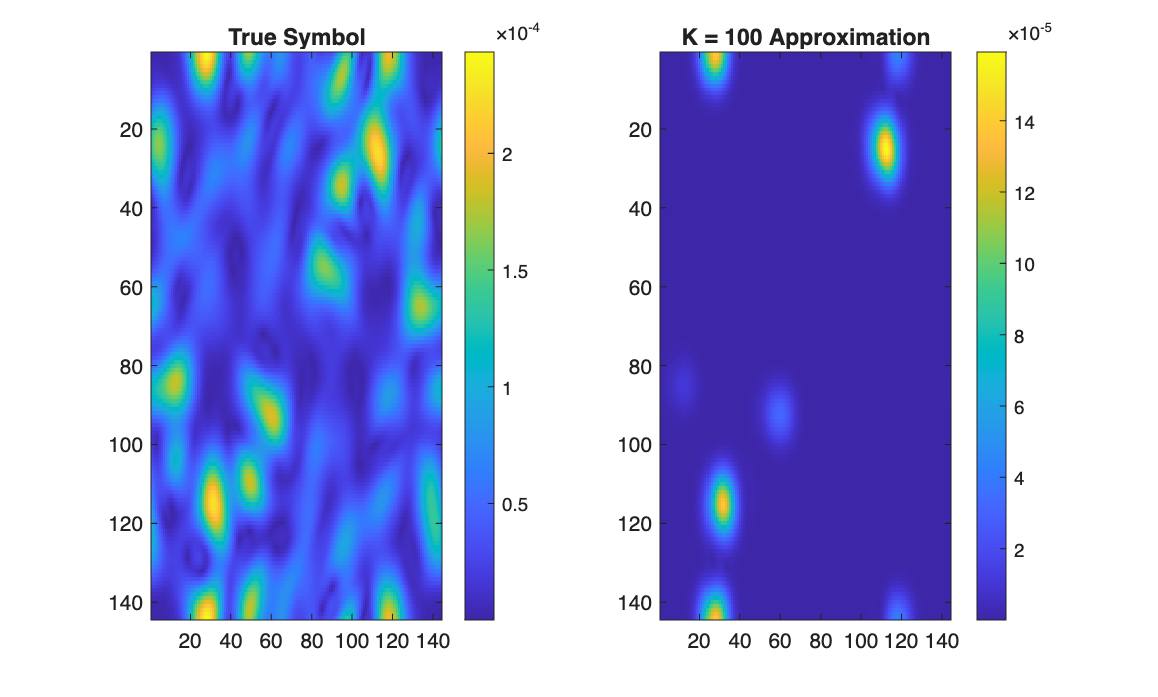}}
    \subfigure[K=200]{\includegraphics[width=0.3\linewidth]{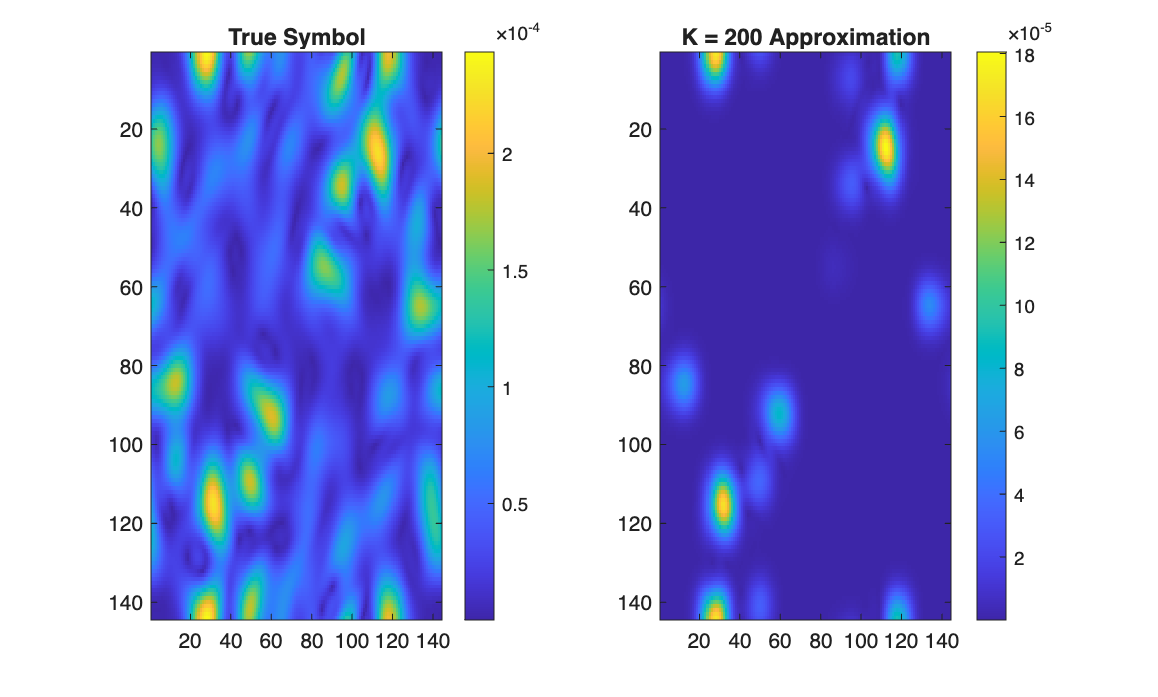}}
    
    \subfigure[K=500]{\includegraphics[width=0.3\linewidth]{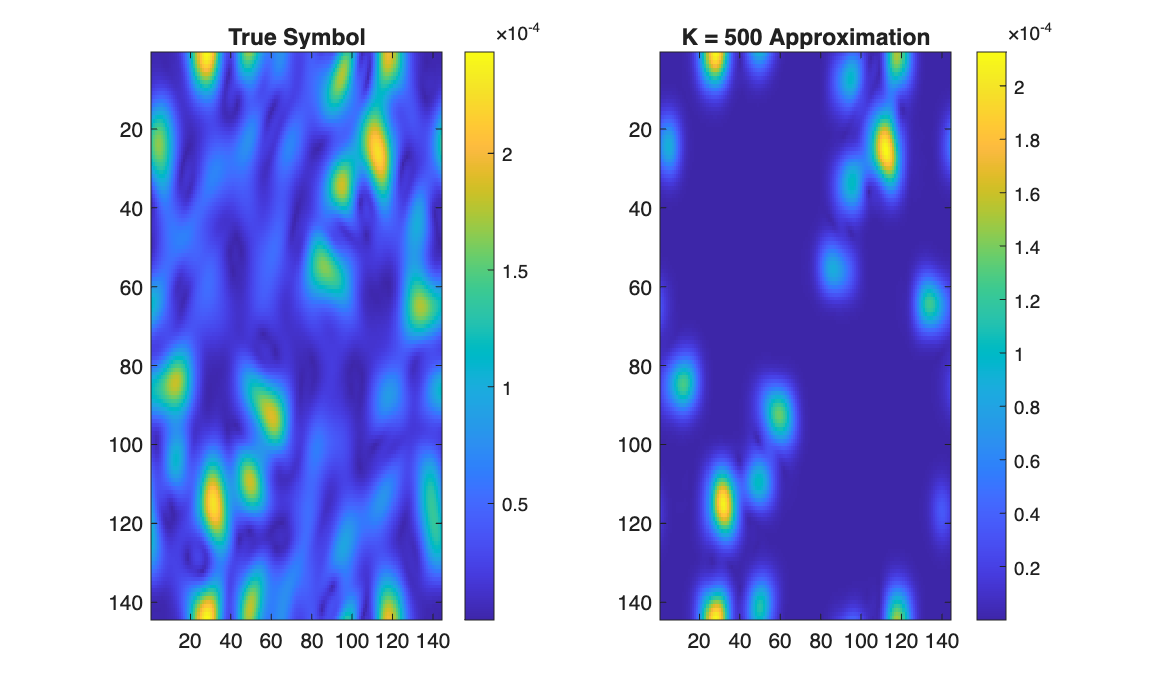}}
    \subfigure[K=1000]{\includegraphics[width=0.3\linewidth]{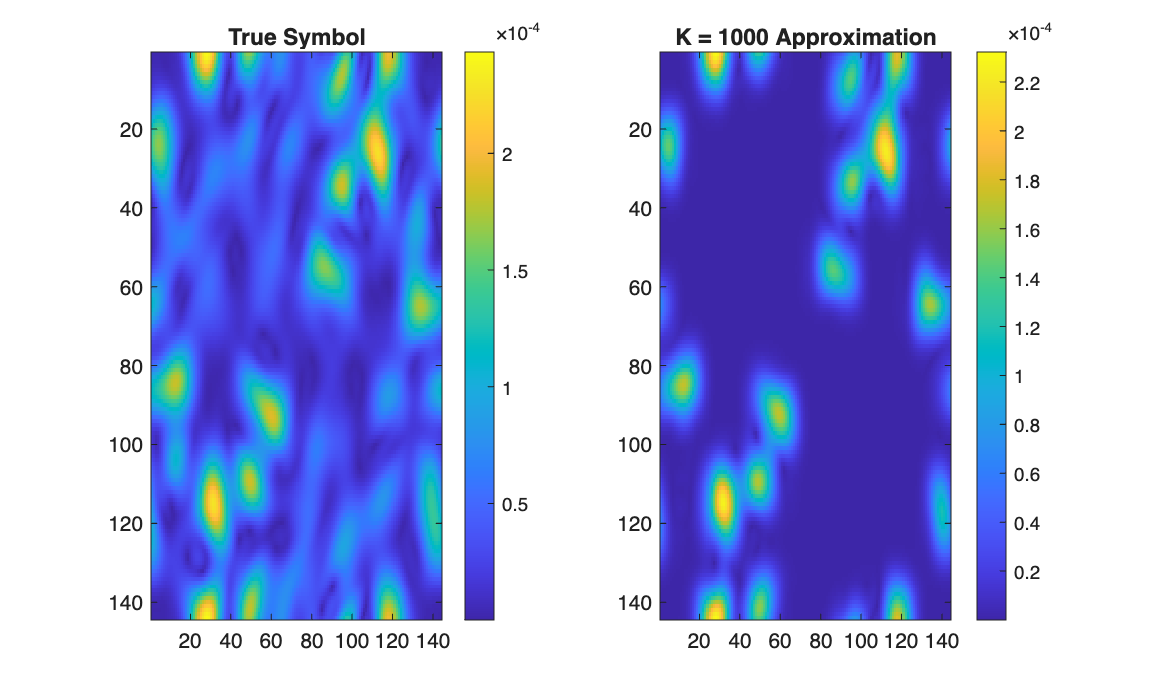}}
    \subfigure[K=4000]{\includegraphics[width=0.3\linewidth]{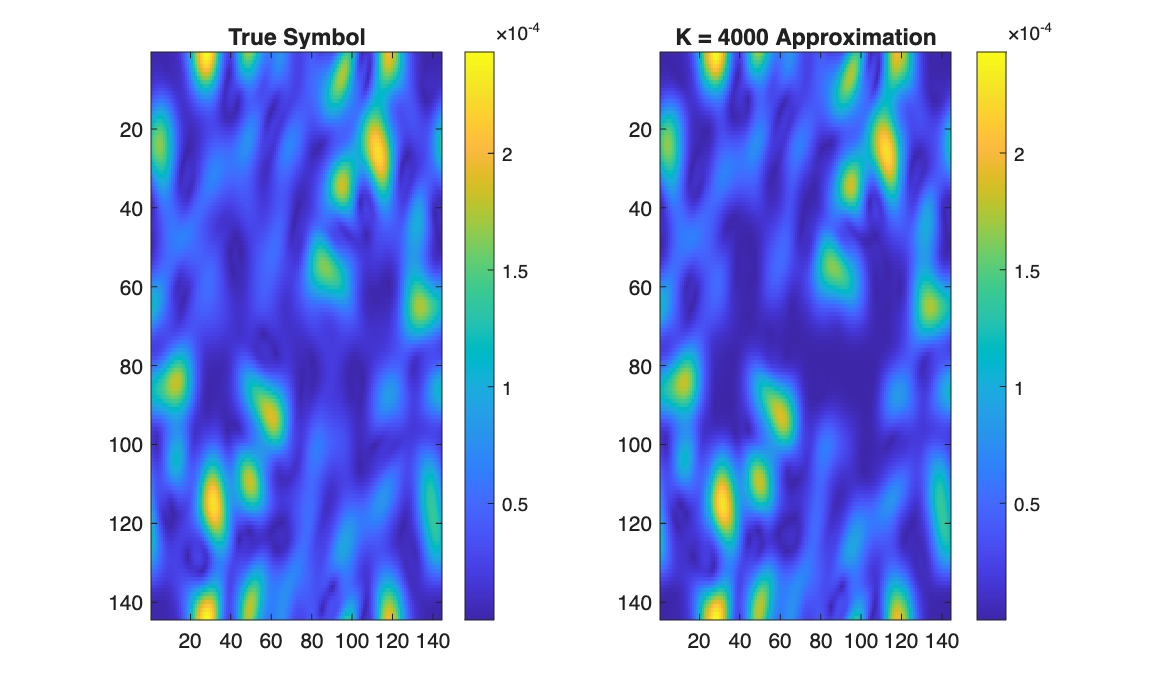}}

    \caption{Approximation results using best \(K\)-term expansions of the operator STFT.}
    \label{fig:K_term_approximations}
\end{figure}

\subsection{Time-frequency localization operators}We look at the sparse approximation of TF-localization operators discussed in Example~\ref{Ex2} next.
For this class of operators, we discuss the concept of denoising via Sparse m-term Approximation. The reasoning is as follows: TF-localization operators such as \( H_m \) can often be well-approximated using only a small number of elementary time-frequency shifted atoms, i.e., it exhibits a sparse representation in a structured dictionary  with atoms of the form \( \Phi^* \pi(\lambda)^* H_m \), where \( \Phi \) is the analysis operator and \( \pi(\lambda) \) denotes a TF shift indexed by lattice parameters.

In other words, a truncated approximation—retaining only the top \( K \) coefficients in Frobenius norm—can capture the essential structure of \( F \) with high fidelity. 

When noise is added to \( H_m \), which may model noisy measurements for example,  its energy spreads across many components in the TF domain, and the resulting operator becomes less sparse. However, since the noise contributions are typically incoherent with respect to the TF structure of the clean operator, they manifest as many small-magnitude coefficients. By thresholding and retaining only the \( K \) largest coefficients, we suppress these small, noise-dominated components. The reconstructed operator, built from the most significant time-frequency atoms, therefore retains the essential TF features of the original \( F \) while discarding high-frequency or incoherent noise—effectively acting as a denoising filter on the operator level. 

\begin{figure}[htbp]
    \centering
    \includegraphics[width=0.6\linewidth]{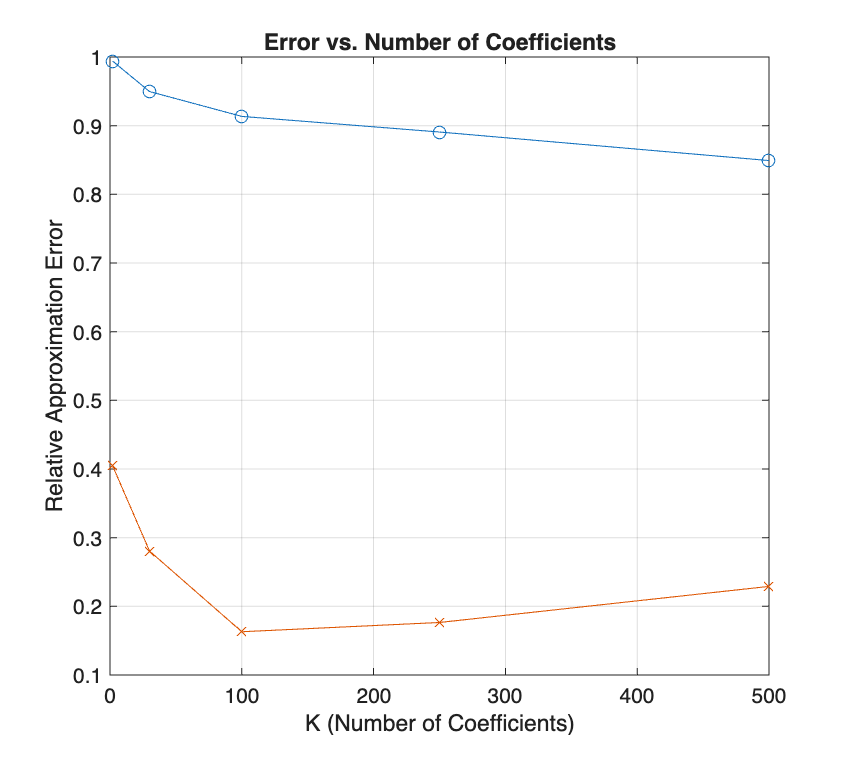}
    \caption{Relative approximation error versus the number of coefficients \( K \) used in the reconstruction of a noisy time-frequency localization operator. 
    The blue curve (with circles) corresponds to the approximation error for the \textit{noisy} version of $H_M$, while the orange curve (with crosses) shows the error in approximating the clean operator using only the top \( K \) coefficients in the time-frequency expansion. 
    As \( K \) increases, the error in the clean component decreases initially but then increases slightly, illustrating a trade-off: including too many small (noise-dominated) coefficients begins to reintroduce noise. The orange curve demonstrates that moderate values of \( K \) lead to effective denoising and accurate recovery.}
    \label{fig:GMerror_vs_K}
\end{figure}

In the presented example, 500 components are used.

\begin{lstlisting}[language=Matlab, caption={MATLAB code to construct a mixed-state operator}, basicstyle=\ttfamily\small, keywordstyle=\color{blue}]
F = zeros(size(F)); % Initialize with zeros 
for k = 1:250
    h = rotmod(g, randi(N), randi(N)); 
    F = F + h' * h;  
    h = rotmod(g, round(randn * 20), round(randn * 20)) * 0.3;
    F = F + h' * h;
end
\end{lstlisting}

{\it Results are visualized in  Figure~\ref{Fig:MixedStateApp}}, where  both the operator symbol (top-left) and its sparsified approximation (top-right), along with their actions on white noise inputs (bottom row) are depicted.

\begin{figure}[h]
    \centering
\includegraphics[width=0.8\textwidth]{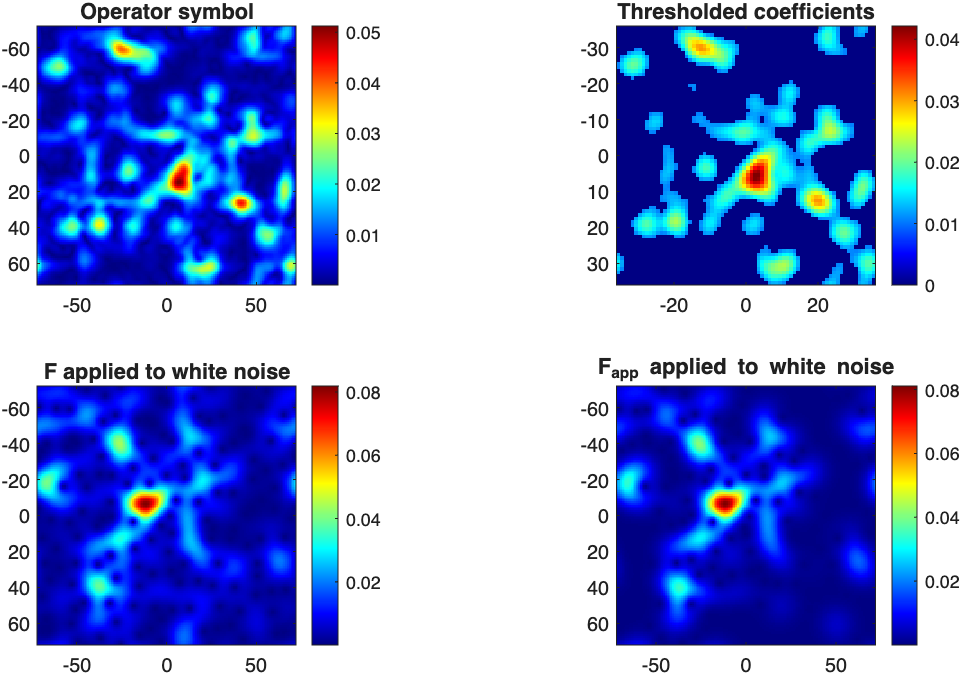}
\caption{Approximation and action of a mixed-state operator.}\label{Fig:MixedStateApp}
\end{figure}

\begin{itemize}
    \item \textbf{Top-left (Operator symbol):} This shows the Wigner-Weyl or phase-space representation of the operator \( F \). It reveals the localized features (e.g., peaks or coherent contributions) that define the structure of the mixed state.
    
    \item \textbf{Top-right (Thresholded coefficients):} This is a sparsified version of the operator coefficients, obtained by applying a threshold to eliminate insignificant components. The goal is to approximate \( F \) with fewer active coefficients.
    
    \item \textbf{Bottom-left  (\( F \) applied to white noise):} This is the ground truth — the result of applying the full operator \( F \) to white noise. Comparison with the bottom-right plot evaluates the fidelity of the approximation.
    \item \textbf{Bottom-right (\( F_{\text{app}} \) applied to white noise):} This subplot shows the effect of the sparsified operator \( F_{\text{app}} \) when applied to white noise. It tests whether the main structural features of \( F \) are preserved under approximation.
   
\end{itemize}
The 
{\it  relative approximation error}
is $0.2339$ with $78\%$ of the coefficients in the reconstruction set to zero.

\subsection{Decay of best $m$-term approximation}
The  experiments in this section are conducted with operators of dimension $N = 144$. The TF analysis is performed on a regular lattice with spacing parameters $a = b = 4$ or
$a = b = 2$. The approximation procedure involves computing the operator short-time Fourier transform (STFT) coefficients with respect to an analysis operator $\Phi$, which is taken from various classes,  retaining  $m$ most significant terms, and then reconstructing an approximation $\hat{F}$ using these terms.

To empirically validate Lemma 4.8 and Remark 4.10, we conduct a systematic evaluation of best-$m$-term operator approximations. In each trial, we fix the number of coefficients $m$ and construct an approximation $\hat{F}_m$ to an operator $F$ using only the $m$ largest analysis coefficients in Frobenius norm. The analysis is performed via localized Gabor-type frames generated over lattices $\Lambda \subset \mathbb{Z}^2$.

The relative approximation error is computed as
\[
\mathrm{AppErr}_K = \frac{\|F - \hat{F}_m\|_{\mathrm{F}}}{\|F\|_{\mathrm{F}}},
\]
and the coefficient usage ratio is given by $K / |\Lambda|$. Additionally, we measure the behavioral similarity using white noise probes:
\[
\mathrm{WNE}_m = \mathbb{E}_{\eta}\left[ \frac{\|F\eta - \hat{F}_m\eta\|}{\|F\eta\|} \right],
\]
where $\eta \sim \mathcal{N}(0, I)$.

\subsubsection{Test Categories}

We consider ten test scenarios:

\begin{itemize}
    \item[\textbf{Test 1}] Infinite-rank operator with infinite-rank analysis.
    \item[\textbf{Test 2}] Infinite-rank operator with rank-2 analysis.
    \item[\textbf{Test 3}] Infinite-rank operator with rank-1 analysis.
    \item[\textbf{Test 4}] Infinite-rank operator with rank-4 analysis.
    \item[\textbf{Test 5}] Infinite-rank operator with rank-6 analysis.
    \item[\textbf{Test 6}] Infinite-rank operator with analysis from its six leading eigenfunctions.
    \item[\textbf{Test 7}] Mixed-state operator with infinite-rank analysis.
    \item[\textbf{Test 8}] Mixed-state operator with rank-2 analysis.
    \item[\textbf{Test 9}] Mixed-state operator with rank-1 analysis.
    \item[\textbf{Test 10}] Random operator with rank-1 analysis (baseline).
\end{itemize}

Each test reports approximation errors, sparsity, and additional norm-based diagnostics.

We implemented a best $m$ term approximation of operators using Gabor-type localized analysis. 
For each of the nine test configurations, we recorded:
the relative approximation error 
    \[
    \text{AppErr}(K) = \frac{\|F - \hat{F}_m\|_{\mathrm{F}}}{\|F\|_{\mathrm{F}}},
    \]
and the relative size of the coefficient set $\frac{K}{|\Lambda|}$. 

\subsubsection{Decay of Approximation Error}
The decay of $\text{AppErr}(K)$ depends on the sparsity of the operator in the chosen representation. 
If the sequence of coefficients belongs to $\ell^p$ with $0<p<2$, the best $m$ term approximation error satisfies
\[
\|F - \hat{F}_m\|_{\mathrm{HS}} \leq C \, K^{\frac{1}{2} - \frac{1}{p}} \, \|c\|_{\ell^p}.
\]
Empirically, the error decays approximately as a power law in $K$, showing rapid decrease for small $K$ and slower convergence as $K$ approaches the total number of coefficients.

\begin{figure}[ht]
    \centering
    \includegraphics[width=1.1\textwidth]{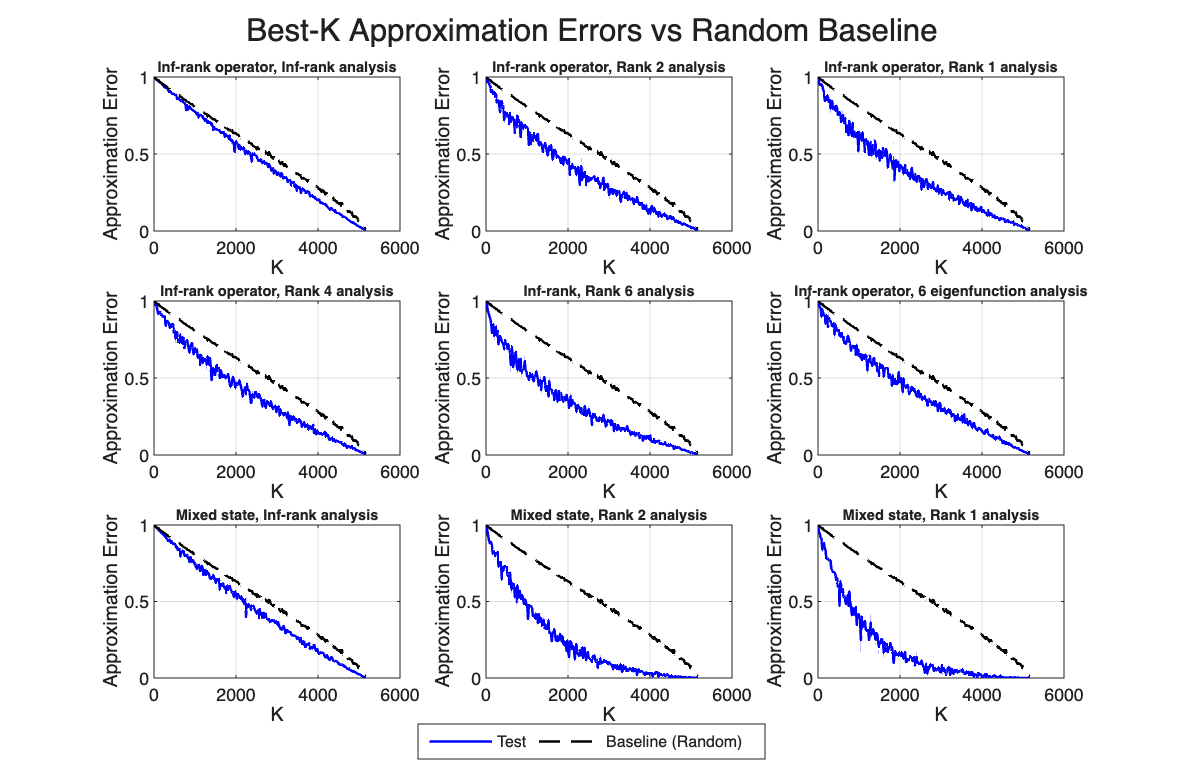}
    \caption{
        Absolute approximation errors $\mathrm{ER}(K)$ for the nine test scenarios 
        compared against the baseline given by a random operator (Test~10).
        Each subplot shows the decay of the approximation error as the number 
        of retained terms $K$ increases.
        The dashed black curve indicates the error decay for the random operator, 
        which is expected to decrease only linearly with $K$.
        Faster decay indicates that the structure of the operator (e.g., 
        low-rank or mixed state) is being exploited effectively by the 
        best-$K$ approximation.
    }
    \label{fig:bestk-absolute}
\end{figure}
The ten test setups explore how different operator structures and analysis models 
affect best $m$ term approximation. Tests~1--4 correspond to infinite-rank operators 
with varying analysis ranks (infinite, 2, 1, and 4, respectively), while 
Tests~5--6 employ either eigenfunction-based analysis or a finite-rank 
approximation (rank~6). Tests~7--9 investigate mixed-state operators under 
infinite-, rank-2, and rank-1 analyses. Finally, Test~10 serves as a baseline, 
using a random operator with rank-1 analysis, against which all other cases are 
benchmarked. The resulting approximation error decay curves are shown in 
Fig.~\ref{fig:bestk-absolute}, where structured operators exhibit markedly faster 
error reduction compared to the random baseline. 

\subsection{ Decay of Approximation Error for operators with smooth symbol}

To specifically illustrate the theoretical findings on sparse approximation in operator coorbit spaces, 
we conducted a numerical study of best-$m$ term approximation for infinite-rank operators 
generated by polynomially weighted Fourier--Wigner multipliers. 
We use two different types of analysis operators to further illustrate the flexibility of the theory as outlined in Theorem~\ref{mainchar2}, more specifically, a rank-one analysis operator based on a Gaussian window and, alternatively, operators with higher rank, but similar asymptotic localization.  

\subsubsection{Experimental Setup}
We fixed the ambient dimension $N = 144$ and considered redundancy parameters $a=b=3$, 
so that the total number of operator coefficients is 
$N^2/(ab) = 2304$. 
As analysis window we first used the projection onto a normalized Gaussian, see Figure~\ref{fig:Bestm}, and second the projection onto the subspace generated by a random mix of $6$ time-frequency shifted Gaussian windows, see Figure~\ref{fig:Bestm6}. 
The operator under investigation was constructed by multiplying a random Gaussian field 
in the Fourier-Wigner  domain with polynomially decreasing weights
\[
w_\alpha(z) = (1+|z|)^{-\alpha}, 
\quad z \in \{-N/2,\dots,N/2-1\},
\]
followed by an inverse Fourier transform. 

For each decay parameter $\alpha = 1,\dots,9$, and additionally for a Gaussian-decaying 
reference case, we applied the best-$m$ term operator approximation. 
The algorithm selects the $m$ largest operator coefficients 
(in absolute value) with respect to the finite-rank analysis mapping.

The resulting approximation errors as functions of $m$ are displayed in 
Figures~\ref{fig:Bestm} and~\ref{fig:Bestm6}. Each subplot corresponds to a fixed value of $\alpha$ 
(polynomial weight order), and the error curves are plotted against the baseline 
case of Gaussian decay. The plots clearly show that stronger polynomial weights 
lead to slower decay of approximation error as $m$ increases, reflecting reduced 
sparsity in the operator representation. In contrast, the Gaussian-weighted baseline 
exhibits the fastest error decay, confirming the sparsifying effect of exponential 
localization in the Fourier--Wigner domain. It is also interesting to see, that the rank-6 analysis operator yields by and large similar results. This observation will be exploited further, in particular, we will study the impact of data-adaptive analysis operators in future work. 

This experiment demonstrates the interplay between decay properties of operator 
symbols and the efficiency of best-$m$ term approximation. In particular, it numerically 
confirms the theoretical prediction that operators with faster decay 
in their Fourier--Wigner transform admit sparser approximations in operator coorbit spaces.

As a final illustration, we show, in Figure~\ref{fig:Approximating_FW_GaussDecay}, the operator kernel of one of the used rank-6 analysis operators $\Phi$ as well as its canonical dual $\widetilde{\Phi}$ are depicted. The lower plots show the approximation quality by means of the corresponding  specific frame, when used for approximating an operator with Gaussian decay in the Fourier-Wigner domain.

\begin{figure}[h]
    \hspace{-1.2cm}
    \includegraphics[width=1.2\textwidth]{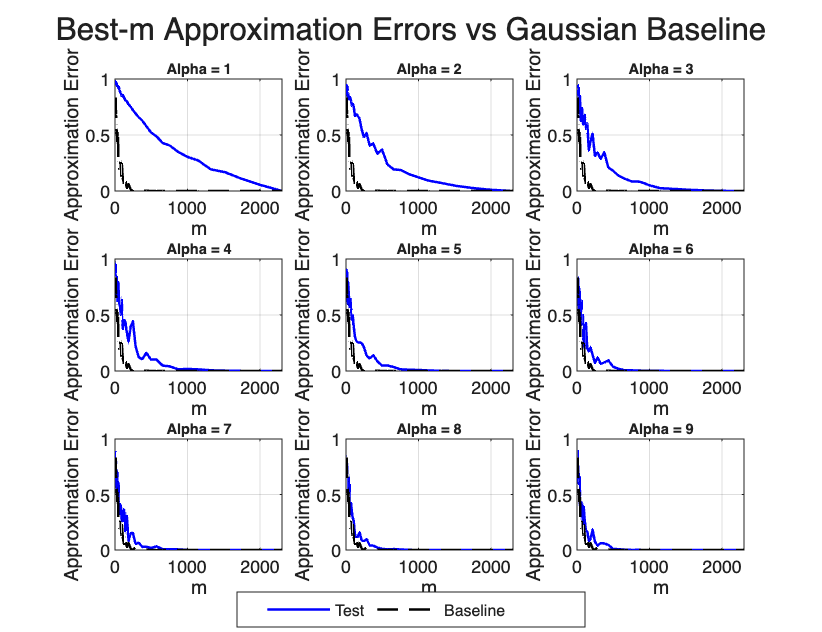}
    \caption{Approximation error curves for best-$m$ term operator approximation.
    Each subplot corresponds to a polynomial decay parameter $\alpha=1,\dots,9$, 
    compared against the Gaussian baseline (dashed). Rank of Analysis Operator $\Phi$ is $1$.}
    \label{fig:Bestm}
\end{figure}

\begin{figure}[h!]
   \hspace{-1.2cm}
    \includegraphics[width=1.2\textwidth]{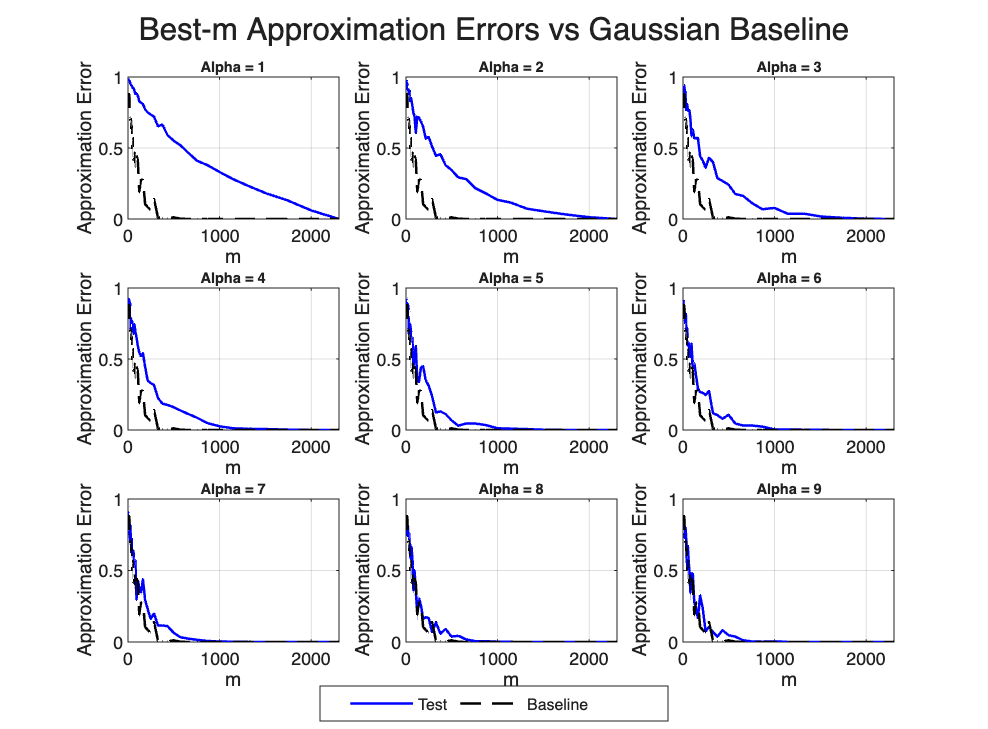}
    \caption{Approximation error curves for best-$m$ term operator approximation.
    Each subplot corresponds to a polynomial decay parameter $\alpha=1,\dots,9$, 
    compared against the Gaussian baseline (dashed). Rank of Analysis Operator $\Phi$ is $6$}
    \label{fig:Bestm6}
\end{figure}

\begin{figure}[t]
    \centering
    \includegraphics[width=1.2\textwidth]{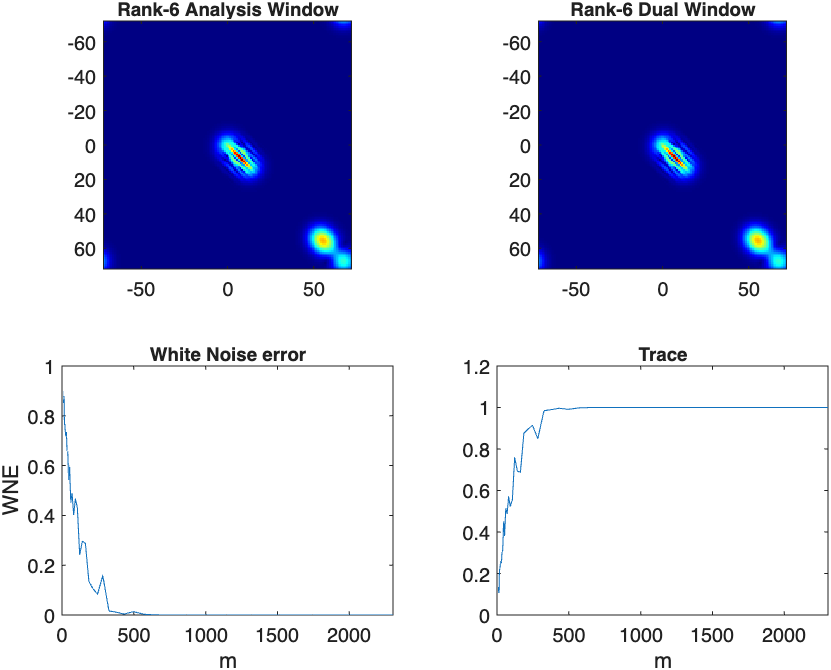}
    \caption{Best-$m$ term approximation of   operator with Gaussian decay in spreading domain. 
     Rank of Analysis Operator $\Phi$, shown in upper left plot,  is $6$, upper right plot shows canonical dual operator. Its rank is $7$. Lower plots show the approximation error for corresponding   frame in approximating an operator with Gaussian decay in the Fourier-Wigner domain. Lower left plot shows white noise probes error and lower right plot shows the trace of the approximation.   }
    \label{fig:Approximating_FW_GaussDecay}
\end{figure}

\section*{Acknowledgments}
    
This research was funded in whole or in part by the Austrian Science Fund (FWF) 10.55776/P34624. For open access purposes, the author has applied a CC BY public copyright license to any author accepted manuscript version arising from this submission. 

L.K. thanks F.L. and NTNU Trondheim for their great hospitality.

\bibliography{biblioall}

\begin{thebibliography}{10}

\bibitem{CORDERO200429}
E.~Cordero and K.~Gröchenig.
\newblock Localization of frames {I}{I}.
\newblock {\em Applied and Computational Harmonic Analysis}, 17(1):29--47, 2004.
\newblock Special Issue: Frames in Harmonic Analysis, Part 1.

\bibitem{DeVore1993ConstructiveA}
R.~A. DeVore and G.~G. Lorentz.
\newblock Constructive approximation.
\newblock In {\em Grundlehren der mathematischen Wissenschaften}, 1993.

\bibitem{DeVore1996}
R.~A. DeVore and V.~N. Temlyakov.
\newblock Some remarks on greedy algorithms.
\newblock {\em Advances in Computational Mathematics}, 5(1):173--187, Dec 1996.

\bibitem{doelumcskr24}
M.~D\"orfler, F.~Luef, H.~McNulty, and E.~Skrettingland.
\newblock Time-frequency analysis and coorbit spaces of operators.
\newblock {\em Journal of Mathematical Analysis and Applications}, 534, 2024.

\bibitem{Dorf21}
M.~D\"orfler, F.~Luef, and E.~Skrettingland.
\newblock Local structure and effective dimensionality of time series data sets.
\newblock {\em Applied and Computational Harmonic Analysis}, 73:101692, 2024.

\bibitem{duffschaef1}
R.~J. Duffin and A.~C. Schaeffer.
\newblock A class of nonharmonic {F}ourier series.
\newblock {\em Trans. Amer. Math. Soc.}, 72:341--366, 1952.

\bibitem{fashazlo25}
M.~Faulhuber, I.~Shafkulovska, and I.~Zlotnikov.
\newblock On the frame property of {H}ermite functions and exploration of their frame sets.
\newblock {\em Journal of Fourier Analysis and Applications}, 31, 04 2025.

\bibitem{feich81}
H.~G. Feichtinger.
\newblock On a new {Segal} algebra.
\newblock {\em Monatshefte für Mathematik}, 92, 1981.

\bibitem{forngroech05}
M.~Fornasier and K.~Gr{\"o}chenig.
\newblock Intrinsic localization of frames.
\newblock {\em Constr. Approx.}, 22(3):395--415, 2005.

\bibitem{Gilyen2022}
A.~Gilyén, S.~Lloyd, I.~Marvian, Y.~Quek, and M.~M. Wilde.
\newblock Quantum algorithm for petz recovery channels and pretty good measurements.
\newblock {\em Physical Review Letters}, 128(22):220502, 2022.

\bibitem{grib-nielsen2}
R.~Gribonval and M.~Nielsen.
\newblock Highly sparse representations from dictionaries are unique and independent of the sparseness measure.
\newblock {\em Applied Comput. Harmon. Anal}, 22:335--355, 2007.

\bibitem{gr01}
K.~Gr{\"o}chenig.
\newblock {\em Foundations of Time-Frequency Analysis}.
\newblock Birkh{\"a}user, Boston, 2001.

\bibitem{gr04-1}
K.~{G}r{\"o}chenig.
\newblock {L}ocalization of {F}rames, {B}anach {F}rames, and the {I}nvertibility of the {F}rame {O}perator.
\newblock {\em {J}. {F}ourier {A}nal. {A}ppl.}, 10(2):105--132, 2004.

\bibitem{GR14}
K.~Gr{\"o}chenig.
\newblock The mystery of {G}abor frames.
\newblock {\em Journal of Fourier Analysis and Applications}, 20:865--895, 08 2014.

\bibitem{gröstö13}
K.~Gr{\"o}chenig and J.~St{\"o}ckler.
\newblock Gabor frames and totally positive functions.
\newblock {\em Duke Mathematical Journal}, 162(6):1003--1031, 2013.

\bibitem{grochenigtfa}
K.~Gröchenig.
\newblock {\em Foundations of Time-Frequency Analysis}.
\newblock Springer, 2001.

\bibitem{gr06weightsinTFA}
K.~Gröchenig.
\newblock {\em Weight functions in time-frequency analysis}, chapter~15, pages 343--366.
\newblock American Mathematical Society, 11 2007.

\bibitem{haubakö25}
N.~Hauschka, P.~Balazs, and L.~Köhldorfer.
\newblock Banach distribution spaces for a {H}ilbert space.
\newblock In {\em 2025 International Conference on Sampling Theory and Applications (SampTA)}, pages 1--5, 2025.

\bibitem{Hausladen1996}
P.~Hausladen, R.~Jozsa, B.~Schumacher, M.~Westmoreland, and W.~Wootters.
\newblock Classical information capacity of a quantum channel.
\newblock {\em Physical Review A}, 54(3):1869--1876, 1996.

\bibitem{HausladenWootters1994}
P.~Hausladen and W.~K. Wootters.
\newblock A "pretty good" measurement for distinguishing quantum states.
\newblock {\em Journal of Modern Optics}, 41(12):2385--2390, 1994.

\bibitem{hol22}
J.~Holb\"ock.
\newblock Localized frames and applications.
\newblock Master's thesis, University of Vienna, 2022.

\bibitem{KuengPreskil2022}
H.-Y. Huang, R.~K{\"u}ng, G.~Torlai, V.~Albert, and J.~Preskill.
\newblock Provably efficient machine learning for quantum many-body problems.
\newblock {\em Science}, 377(6613), 2022.

\bibitem{HyNeVeWe16}
T.~Hytönen, J.~van Neerven, M.~Veraar, and L.~Weis.
\newblock {\em Analysis in Banach Spaces, Volume I: Martingales and Littlewood-Paley Theory}.
\newblock Springer, 12 2016.

\bibitem{jakob18}
M.~Jakobsen.
\newblock On a (no longer) new {Segal} algebra — a review of the {Feichtinger} algebra.
\newblock {\em Journal of Fourier Analysis and Applications}, 24(6), 2018.

\bibitem{ja96}
A.~Janssen.
\newblock Some {W}eyl-{H}eisenberg frame bound calculations.
\newblock {\em Indagationes Mathematicae}, 7(2):165--183, 1996.

\bibitem{ja03}
A.~Janssen.
\newblock On generating tight {G}abor frames at critical density.
\newblock {\em Journal of Fourier Analysis and Applications}, 9:175--214, 03 2003.

\bibitem{jastro02}
A.~Janssen and T.~Strohmer.
\newblock Hyperbolic secants yield {G}abor frames.
\newblock {\em Applied and Computational Harmonic Analysis}, 12(2):259--267, 2002.

\bibitem{KaftalLarsonZhang-OpValuedFrames}
V.~Kaftal, D.~R. Larson, and S.~Zhang.
\newblock Operator-valued frames.
\newblock {\em Trans. Am. Math. Soc.}, 361(12):6349--6385, 2009.

\bibitem{KeylKiukasWerner2015}
M.~Keyl, J.~Kiukas, and R.~F. Werner.
\newblock Schwartz operators.
\newblock {\em Reviews in Mathematical Physics}, 28(3):1630001, 2016.
\newblock arXiv:1503.04086 [math-ph]; non-commutative analog of Schwartz functions, properties, tempered distributions, etc.

\bibitem{KotowskiOszmaniec2025}
M.~Kotowski and M.~Oszmaniec.
\newblock Pretty‐good simulation of all quantum measurements by projective measurements.
\newblock {\em arXiv preprint arXiv:2501.09339}, 2025.
\newblock Preprint.

\bibitem{kutpatphi17}
G.~Kutyniok, V.~Paternostro, and F.~Philipp.
\newblock The effect of perturbations of frame sequences and fusion frames on their duals.
\newblock {\em Oper. Matrices}, 11:301--336, 2017.

\bibitem{köbaLOC25}
L.~Köhldorfer and P.~Balazs.
\newblock Localization of operator-valued frames. \url{https://arxiv.org/abs/2503.24170}, 2025.

\bibitem{köbasampta25}
L.~Köhldorfer and P.~Balazs.
\newblock On the inverse-closedness of operator-valued matrices with polynomial off-diagonal decay.
\newblock In {\em 2025 International Conference on Sampling Theory and Applications (SampTA)}, pages 1--5, 2025.

\bibitem{koeba25}
L.~Köhldorfer and P.~Balazs.
\newblock Wiener pairs of {B}anach algebras of operator-valued matrices.
\newblock {\em Journal of Mathematical Analysis and Applications}, 549(2):129525, 2025.

\bibitem{skrett18}
F.~Luef and E.~Skrettingland.
\newblock Convolutions for localization operators.
\newblock {\em Journal de Mathématiques Pures et Appliquées}, 25(4), 2018.

\bibitem{luxu23}
F.~Luef and X.~Wang.
\newblock Gaussian {G}abor frames, {S}eshadri constants and generalized {B}user–{S}arnak invariants.
\newblock {\em Geometric and Functional Analysis}, 33:1--46, 04 2023.

\bibitem{Lyu92}
Y.~Lyubarskii.
\newblock {\em Frames in the Bargmann space of entire functions}, pages 167--180.
\newblock American Mathematical Society, Providence, 10 1992.

\bibitem{sei92}
K.~Seip.
\newblock Density theorems for sampling and interpolation in the {B}argmann-{F}ock space {I}.
\newblock {\em Bulletin of the American Mathematical Society}, 26:91--106, 05 1992.

\bibitem{seiwal92}
K.~Seip and R.~Wallstén.
\newblock Density theorems for sampling and interpolation in the {B}argmann-{F}ock space {II}.
\newblock {\em Journal für die reine und angewandte Mathematik}, 429:107--114, 1992.

\bibitem{skrett21}
E.~Skrettingland.
\newblock On {Gabor} g-frames and {Fourier} series of operators.
\newblock {\em Studia Mathematica}, 259(1), 2021.

\bibitem{Ste51}
S.~B. Stechkin.
\newblock On absolute convergence of orthogonal series.
\newblock {\em Dokl. Akad. Nauk SSSR}, 102:37 -- 40, 1955.

\bibitem{sun06}
W.~Sun.
\newblock G-frames and g-{Riesz} bases.
\newblock {\em Journal of Mathematical Analysis and Applications}, 322, 2006.

\bibitem{taovu12}
T.~Tao and V.~Vu.
\newblock Random covariance matrices: Universality of local statistics of eigenvalues.
\newblock {\em Annals of Probability}, 40(3):1285--1315, May 2012.

\bibitem{raul11}
T.~{U}llrich and H.~{R}auhut.
\newblock {G}eneralized coorbit space theory and inhomogeneous function spaces of {B}esov-{L}izorkin-{T}riebel type.
\newblock {\em {J}. {F}unct. {A}nal.}, 11:3299--3362, 2011.

\bibitem{HanLiWang-OpValGaborFrames}
J.~Wang, P.~Li, and D.~Han.
\newblock The density theorem for operator-valued frames via square-integrable representations of locally compact groups.
\newblock {\em J. Fourier Anal. Appl.}, 30(5):29, 2024.
\newblock Id/No 47.

\end{thebibliography}
\bibliographystyle{abbrv}

\Addresses

\end{document}